\documentclass[11pt]{article} \topmargin=0pt 
\usepackage[center]{caption2}
\usepackage{float}
\usepackage{epsfig}
\usepackage{amssymb}
\usepackage{amsmath}
\usepackage{graphicx, color, subfigure}
\usepackage{graphicx}
\usepackage{epstopdf}
\usepackage{amsmath,amssymb}
\usepackage{multirow}
\usepackage{fancyhdr}
\usepackage[english]{babel}
\usepackage{epstopdf}
\usepackage{fancybox}
\usepackage{bm}
\usepackage{amssymb}
\usepackage{amsmath}
\usepackage{verbatim}
\usepackage{graphicx, color, subfigure,amsthm}
\usepackage{xcolor}
\usepackage{colortbl}
\usepackage{mathrsfs}
\newcommand{\be}{\begin{equation}}
\newcommand{\ee}{\end{equation}}

\newtheorem{theorem}{Theorem}[section]
\newtheorem{lemma}{Lemma}[section]

\newtheorem{definition}{Definition}[section]

\newtheorem{remark}{Remark}[section]
\newtheorem{example}{Example}[section]
\begin{document}
\baselineskip 6mm
\renewcommand{\refname}{\begin{flushleft}{{\bf\small
Reference}}\end{flushleft}}
\title {{\bf Bound-preserving discontinuous Galerkin method for compressible miscible displacement in porous media}
 \thanks{\small  Supported by National Natural Science Foundation of China (11571367 and 11601536) and Michigan
Technological University Research Excellence Fund Scholarship and Creativity Grant 1605052}
% \thanks{\small Correspondence to:  Yang Yang, Department of Mathematical Sciences, Michigan Technological University, Houghton, MI 49931, USA (Email: yyang7@mtu.edu)}
}
 \vskip 30mm
\author { Hui Guo\thanks{College of Science, China University of Petroleum, Qingdao 266580, China. E-mail: sdugh@163.com}\quad\quad\quad\quad Yang Yang\thanks{Michigan Technological University, Houghton, MI 49931, USA. E-mail:yyang7@mtu.edu}
}
\date{}
\maketitle
\begin{center}
\begin{minipage}{13.4cm}
{\bf \small Abstract:}
{\small
In this paper, we develop bound-preserving discontinuous Galerkin (DG) methods for the coupled system of compressible miscible displacement problems. We consider the problem with two components and the (volumetric) concentration of the $i$th component of the fluid mixture, $c_i$, should be between $0$ and $1$. However, $c_i$ does not satisfy the maximum principle. Therefore, the numerical techniques introduced in (X. Zhang and C.-W. Shu, Journal of Computational Physics, 229 (2010), 3091-3120) cannot be applied directly. The main idea is to apply the positivity-preserving techniques to both $c_1$ and $c_2$, respectively and enforce $c_1+c_2=1$ simultaneously to obtain physically relevant approximations. By doing so, we have to treat the time derivative of the pressure $dp/dt$ as a source in the concentration equation. Moreover, $c_i's$ are not the conservative variables, as a result, the classical bound-preserving limiter in (X. Zhang and C.-W. Shu, Journal of Computational Physics, 229 (2010), 3091-3120) cannot be applied. Therefore, another limiter will be introduced. Numerical experiments will be given to demonstrate the accuracy in $L^\infty$-norm and good performance of the numerical technique.
}
%{\small \bf AMS(2000) Subject Classifications: } {\small 65M15, 65M60.}\  \  \ \ \  \
\end{minipage}
\end{center}
\section{Introduction}
\label{sec1}

Numerical modeling of miscible displacements in porous media is important and interesting in oil recovery and environmental pollution problem. The miscible displacement problem is described by a coupled system of nonlinear partial differential equations. The need for accurate solutions to the coupled equations challenges numerical analysts to design new methods. In this paper, we study the compressible miscible displacements in porous media on computational domain $\Omega=[0,2\pi]\times[0,2\pi]$
\begin{equation}\label{origin}
\begin{array}{ll}
d(c)p_t+\nabla\cdot\mathbf{u}=d(c)p_t-\nabla\cdot(\frac{\kappa(x,y)}{\mu(c)}\nabla p)=q,&(x,y)\in\Omega,\quad0< t\le T,\\
\phi c_t+b(c) p_t+\mathbf{u}\cdot\nabla c-\nabla\cdot({\bf D}\nabla c)=(\tilde{c}-c)q,& (x,y)\in\Omega,\quad0< t\le T,
\end{array}
\end{equation}
as well as its one-dimensional version. In \eqref{origin} the dependent variables $p$, ${\bf u}$ and $c$ are the pressure in the fluid mixture, the Darcy velocity of the mixture (volume flowing across a unit across-section per unit time), and the concentration of interested species measured in amount of species per unit volume of the fluid mixture, respectively. $\phi$ and $\kappa$ are the porosity and the permeability of the rock, respectively. $\mu$ is the concentration-dependent viscosity. $q$ is the external volumetric flow rate, and $\tilde{c}$ is the concentration of the fluid in the external flow. $\tilde{c}$ must be specified at points where injection ($q>0$) takes place, and is assumed to be equal to $c$ at production points $(q<0)$. The diffusion coefficient ${\bf D}$ arises from two aspects: molecular diffusion, which is rather small for field-scale problems, and dispersion, which is velocity-dependent, in the petroleum engineering literature. Its form is
\begin{equation}
{\bf D}=\left(\begin{array}{cc}D_{11}&D_{12}\\D_{21}&D_{22}\end{array}\right)=\phi(x,y)(d_{mol}{\bf I}+d_{long}|{\bf u}|{\bf E}+d_{tran}|{\bf u}|{\bf E}^{\bot}),
\end{equation}
where ${\bf E}$, a $2\times2$ matrix, represents the orthogonal projection along the velocity vector and is given by
$${\bf E}=(e_{ij}({\bf u}))=\left(\frac{u_{i}u_{j}}{|{\bf u}|^{2}}\right),\quad {\bf u}=(u_1, u_2),$$
and ${\bf E}^{\bot}={\bf I}-{\bf E}$ is the orthogonal complement. The diffusion coefficient $d_{long}$ measures the dispersion in the direction of the flow and $d_{tran}$ shows that transverse to the flow. To ensure the stability of the scheme, in almost all of the previous works ${\bf D}$ is assumed to be strictly positive definite. In this paper, we assume ${\bf D}$ to be positive semidefinite. Therefore, we have $D_{11}\geq0$, $D_{22}\geq0$ and $D_{12}=D_{21}$. In the numerical experiments, we also choose ${\bf D}={\bf 0}$ to challenge the algorithm. Moreover, the pressure is uniquely determined up to a constant, thus we assume $\int_{\Omega} pdxdy=0$ at $t=0$. However, this assumption is not essential. In this paper, we consider a two component displacement only. The coefficients can be stated as follows:
$$
\begin{array}{ll}
c=c_{1}=1-c_{2},&a(c)=a(x,y,c)=\frac{\mu(c)}{\kappa},\\
b(c)=b(x,y,c)=\phi c_{1}\{z_1-\sum\limits_{j=1}^{2}z_j c_j\},& d(c)=d(x,y,c)=\phi\sum\limits_{j=1}^{2}z_j c_j,
\end{array}
$$
where $c_i$ and $z_i$ are the concentration and the compressibility factor of the $i$th component of the fluid mixture, respectively. In this problem, the Neumann boundary conditions are given as
\begin{equation}\label{boundary2d}
{\bf u}\cdot{\bf n}=0,\quad({\bf D}\nabla c-c{\bf u})\cdot{\bf n}=0,
\end{equation}
where ${\bf n}$ is the outer normal of the boundary $\partial\Omega$. Moreover, the initial solutions are given as
$$
c(x,y,0)=c_{0}(x,y),\quad p(x,y,0)=p_{0}(x,y),\quad (x,y)\in \Omega.
$$

The miscible displacement problem in porous media was first introduced in \cite{Douglas1,Douglas2}, in which the mixed finite element methods were presented. Later, the compressible problem was studied in \cite{Douglas Jr} and the optimal order estimates in $L^2$ and almost optimal order estimates in $L^\infty$ were given in \cite{Chou1991}. Subsequently, many new methods were introduced to solve the compressible miscible displacements, such as finite difference method \cite{Yuan1999,Yuan2003,Yuan2004}, characteristics collocation method \cite{Ma2005}, splitting positive definite mixed element method \cite{Yang2001} and $H^1$-Galerkin mixed method \cite{Chen2010}. Besides the above, in \cite{Wang2000a}, an accurate and efficient simulator is developed for problems with wells. Later, the authors introduced an Eulerian-Lagrangian localized adjoint method (ELLAM) to solve the transport partial differential equation for concentration, while a mixed finite element method (MFEM) to solve the pressure equation \cite{Wang2000}. Recently, discontinuous Galerkin (DG) methods became more popular in solving compressible miscible displacement problem in porous media \cite{Cui2007,Cui2008,Yang2010,Yang2011,Guo2015,Yang2010b}. Some special numerical techniques were introduced to control jumps of numerical approximations as well as the nonlinearality of the convection term. Moreover, the superconvergence results based on the DG methods were also proved in \cite{Yang2011b,Yang2013}. Besides the above, there were also significant works discussing the discontinuous Galerkin methods for incompressible miscible displacements, see e.g. \cite{Bartels,Guo,Kumar,Riviere,Sun2002,Sun2005,Wheeler} and for general porous media flow, see e.g. \cite{Bastian,Ern2009,Ern2010,Sun2005b} and the references therein. However, to the best knowledge of the authors, no previous works focused on the bound-preserving techniques. In many numerical simulations, the approximations of $c$ can be placed out of the interval $[0,1]$. Especially for problems with large gradients, the value of $d(c)$ might be negative, leading to ill-posedness of the problem, and the numerical approximations will blow up. We will use numerical experiments to demonstrate this point in Section \ref{secexample}. In this paper, we will discuss the bound-preserving technique for compressible problems. However, the idea can be extended to incompressible flows with some minor changes.

The DG method gained even greater popularity for good stability, high order accuracy, and flexibility on h-p adaptivity and on complex geometry. Physically bound-preserving high-order numerical methods for conservation laws have been actively studied in the last few years. Most cases, the DG schemes were coupled with suitable trouble cell indicators to avoid the order deterioration in regions where the solutions are smooth while maintaining physical solutions near shocks, e.g. \cite{Dedner,Krivodonova}. In 2010, the genuinely maximum-principle-satisfying high order DG and finite volume schemes were constructed in \cite{zhang1} by Zhang and Shu. Subsequently, this technique has been successfully extended to compressible Euler equations without or with source terms \cite{zhang2,zhang3}, shallow water equations \cite{Xing2010}, and hyperbolic equations with $\delta$-singularities \cite{yynm}. The basic idea is to take the test function to be $1$ in each cell to obtain an equation of the numerical cell average of the target variable, say $r$, and prove the cell average, $\bar{r}$, is within the desired bounds. Then we can apply the bound-preserving limiter to the numerical approximation and construct a new one
\begin{equation}\label{limiter}
\tilde{r}=\bar{r}+\theta(r-\bar{r}),\quad\theta\in[0,1].
\end{equation}
If the problem has only one lower bound zero, the technique is also called positivity-preserving technique. Thanks to the limiter, the whole algorithm were proved to be $L^1$ stable \cite{yyjcp,Qin2016} for some complicated systems. Moreover, the technique does not rely on the trouble cell detector and the limiter keeps the high-order accuracy in regions with smooth solutions for scalar equations \cite{zhang1}. For convection-diffusion equations, in \cite{yifan}, the authors applied the same idea to construct genuinely second-order maximum-principle-satisfying DG methods on unstructured meshes. Subsequently, the ideas have been extended to semilinear parabolic equations with blow-up solutions \cite{Guo2015a} and chemotaxis models \cite{yyima} by using the LDG methods. Recently, the flux limiter \cite{Xu2014,Xiong2015} and third-order maximum-preserving direct DG method \cite{Chen2016} were also introduced. In this paper, we will extend the ideas in \cite{zhang1,yifan} and construct second-order bound-preserving DG methods. However, there are significant differences from previous techniques. First of all, most of the problems in \cite{zhang1,yifan} satisfy maximum principles while the concentration $c$ in \eqref{origin} does not. Therefore, we will split the whole algorithm into two parts. We first rewrite the system into a conservative form and treat $p_t$ as another source in the equation of concentration to obtain the positivity of $c$ by the positivity-preserving technique \cite{zhang2,zhang3}. Then we choose a consistent flux pair (see Definition \ref{def}) in the concentration and pressure equations to obtain the positivity of $1-c$. More precisely, in our analysis, instead of solving $c$ and $p$, we rewrite \eqref{origin} into a system of $c$ and $c_2=1-c$ and enforce $c+c_2=1$ by choosing a consistent flux pair. Secondly, to apply the positivity-preserving technique, we need to numerically approximate the conservative variable $r=\phi c$ instead of $c$. By doing so, the upper bound of $r$ is not a constant and the limiter \eqref{limiter} may fail to work, since such a $\theta$ may not exist. Therefore, we will introduce another limiter to obtain physically relevant numerical approximations. Before we finish the introduction, we would like to point out the main difference between the proposed scheme and the TVD-like schemes. For TVD schemes, the numerical approximations are only first-order accurate at the local extrema. In this paper, our scheme is second-order accurate in $L^\infty$-norm. In Section \ref{secexample}, we will construct analytical solutions to demonstrate the accuracy. To the best knowledge of the authors, this is the first analytical solution available.

The paper is organized as follows. In Section \ref{secldg2d}, we present the DG scheme. The bound-preserving techniques will be given in Sections \ref{secbb}. Analytical solutions will be constructed and some numerical results will be given to demonstrate the good performance of the bound-preserving DG method in Section \ref{secexample}. We will end in Section \ref{sec5} with concluding remarks and remarks for future works.

\section{The DG scheme}
\label{secldg2d}
\setcounter{equation}{0}

We first explain the notations that will be used in this section. We consider rectangular meshes only and the techniques for triangular meshes will be discussed in the future. Let $0=x_{\frac12}<\cdots<x_{N_x+\frac12}=2\pi$ and $0=y_{\frac12}<\cdots<y_{N_y+\frac12}=2\pi$ be the grid points in $x$ and $y$ directions, respectively. Define $I_i=(x_{i-\frac12},x_{i+\frac12})$ and $J_j=(y_{j-\frac12},y_{j+\frac12})$. Let $K_{ij}=I_i\times J_j$ be a partition of $\Omega$ and denote $\Omega_h=\{K_{ij}\}.$ For simplicity, we use $K$ to denote the cell. We use $\Gamma$ for all the cell interfaces, and $\Gamma_0=\Gamma\setminus\partial\Omega$. For any $e\in\Gamma$, denote $|e|$ to be the length of $e$. Moreover, we define ${\bf n}_e={\bf n}_x=(1,0)$ if $e$ is parallel to the $y$-axis while ${\bf n}_e={\bf n}_y=(0,1)$ if $e$ is parallel to the $x$-axis. Further more, we denote $\partial\Omega_+=\{e\in\partial\Omega: {\bf n}_e={\bf n}\}$, where ${\bf n}$ is the unit outer normal of $\partial\Omega$, and $\partial\Omega_-=\partial\Omega\setminus\partial\Omega_+.$ The mesh sizes in the $x$ and $y$ directions are given as $\Delta x_{i}=x_{i+\frac12}-x_{i-\frac12}$ and $\Delta y_j=y_{j+\frac12}-y_{j-\frac12}$, respectively. For simplicity, we assume uniform meshes and denote $\Delta x=\Delta x_i$ and $\Delta y=\Delta y_j.$ However, this assumption is not essential. Following \cite{yifan}, we use second-order DG scheme.
%For rectangular meshes, it is impossible to construct a continuous piecewise linear interpolation of $\phi$ by using $P^1$ polynomials. Therefore,
The finite element space is chosen as
$$
W_h=\{z: z|_{K} \in Q^1(K),\forall K \in \Omega_h\},
$$
where $Q^1(K)$ denotes the space of tensor product of linear polynomials in $K$. Given $v\in W_h$, we denote $v^+_{i-\frac12,j}$, $v^-_{i+\frac12,j}$, $v^+_{i,j-\frac12}$, $v^-_{i,j+\frac12}$ to be the traces of $v$ on the four edges of $K_{ij}$, respectively, and use $[v]=v^+-v^-$ and $\{v\}=\frac12(v^++v^-)$ as the jump and average of $v$ at the cell interfaces, respectively. For simplicity, for any $e\in\partial\Omega_-$, we define $v^-|_e=0$. Similarly, for any $e\in\partial\Omega_+$, we define $v^+|_e=0$.

To construct the DG scheme, we first rewrite the system \eqref{origin} into a conservative form
%{\setlength\arraycolsep{2pt}\begin{eqnarray}
%&&d(c)p_t+u_x=q,\label{origin1d1}\\
%&&a(c)u=-p_x,\label{origin1d2}\\
%&&(\phi c)_t+(uc)_x-(D(u)c_x)_x=\tilde{c}q-\phi cz_1p_t,\label{origin1d3}
%\end{eqnarray}}
{\setlength\arraycolsep{2pt}\begin{eqnarray}
&&d(c)p_t+\nabla\cdot\mathbf{u}=q,\label{origin1}\\
&&a(c){\bf u}=-\nabla p,\label{origin2}\\
&&\phi c_t+b(c) p_t+\nabla\cdot(\mathbf{u}c)-\nabla\cdot({\bf D}\nabla c)=\tilde{c}q-\phi cz_1p_t,\label{origin3}
\end{eqnarray}}
where $a(c)=\frac{\mu(c)}{\kappa}$ and $z_1$ is the compressibility factor of first component of the fluid mixture.

Before constructing the bound-preserving DG scheme, we would like to demonstrate the following main points that are quite different from most of the previous works.
\begin{enumerate}
\item Approximate $r=\phi c$ directly instead of $c$. Due to the existence of $\phi$ in the time derivative in \eqref{origin3}, we cannot simply take the test function to be $1$ to extract the cell averages of $c$.
\item Treat $p_t$ as a source in \eqref{origin3}. We will first solve $p_t$ in \eqref{origin1} and then use the positivity-preserving technique introduced in \cite{zhang2,zhang3} to construct positive numerical approximations of $r$.
\item Choose a consistent flux pair for \eqref{origin1} and \eqref{origin3}. In Section \ref{secbb1d}, we will suitably choose the numerical fluxes and prove $\bar{r}<\bar{\phi}$, where $\bar{r}$ and $\bar{\phi}$ are the cell averages of $r$ and $\phi$, respectively.
\item The classical bound-preserving limiters \eqref{limiter} cannot be applied. For example, take $\phi(x)=2-x^2$ on [-1,1] and $\bar{r}=1.1$. We cannot find any $\theta$ such that $0\leq\tilde{r}(x)\leq\phi(x)$ at $x=\pm1.$
\item  Approximate $\phi$ by a piecewise linear approximation. We denote $\Phi\in W_h$ to be an interpolation of $\phi$ such that in each cell $K$, $\Phi=\phi$ at the four corners, and we denote this interpolation operator to be $\mathcal{I}_1$. It is easy to see that $\Phi$ is a globally continuous function on $\Omega$ and define $\Phi_{i+\frac12,j+\frac12}=\phi(x_{i+\frac12},y_{j+\frac12})$. Since $\Phi$ is a second-order approximation of $\phi$, the usage of $\Phi$ will not kill the accuracy.
\item Introduce a new limiter. A new limiter will be given to keep the numerical cell average and modify the numerical approximations such that $0\leq r\leq\Phi$, % Since $r$ and $\Phi$ are both linear functions in each cell, then with the new limiter we have $0\leq r\leq \Phi$ for all points in the computational domain,
    which further yields $c=\frac{r}\Phi\in[0,1]$.
\end{enumerate}

We also use $p,c,{\bf u}$ as the numerical approximations, then the DG scheme is to find $p,r\in W_h$ and ${\bf u}\in{\bf W}_h=W_h\times W_h$ such that for any $\xi,\zeta\in W_h$ and $\bm{\eta}\in{\bf W}_h$,
{\setlength\arraycolsep{2pt}\begin{eqnarray}
(\tilde{d}(r)p_t,\xi)&=&({\bf u},\nabla\xi)+\sum_{e\in\Gamma_0}\int_e\hat{\bf u}[\xi]\cdot{\bf n}_eds+(q,\xi),\label{sc2d1}\\
(a(c){\bf u},\bm{\eta})&=&(p,\nabla\cdot\bm{\eta})+\sum_{e\in\Gamma}\int_e\hat{p}[\bm{\eta}\cdot{\bf n}_e]ds,\label{sc2d2}\\
(r_t,\zeta)&=&\left({\bf u}c-{\bf D}({\bf u})\nabla c,\nabla \zeta\right)+(\tilde{c}q-rz_1p_t,\zeta)+\sum_{e\in\Gamma_0}\int_e\widehat{{\bf u}c}\cdot{\bf n}_e[\zeta]ds\nonumber\\
&&-\sum_{e\in\Gamma_0}\int_e\left(\left\{{\bf D}({\bf u})\nabla c\cdot{\bf n}_e\right\}[\zeta]+\left\{{\bf D}({\bf u})\nabla \zeta\cdot{\bf n}_e\right\}[c]+\frac{\tilde{\alpha}}{|e|}[c][\zeta]\right)ds,\label{sc2d3}
\end{eqnarray}}
where $$c=\mathcal{I}_1\left\{\frac{r}{\Phi}\right\}, \quad \quad\tilde{d}(r)=z_1r+z_2(\Phi-r), \quad\quad(u,v)=\int_\Omega uv dxdy.$$

\begin{remark}
In \eqref{sc2d1}-\eqref{sc2d3}, the usage of ˜$\tilde{d}(r)$ instead of d(c) is required by the positivity-preserving technique. In the proof of Theorems \ref{thm1d2} and \ref{thm2d2}, we will subtract \eqref{sc2d1} from \eqref{sc2d3} to obtain a new source similar to the one in \eqref{sc2d3}. If we use $d(c)$ in \eqref{sc2d1}, then the source in \eqref{sc2d3} should be changed to $c\Phi z_1p_t$.
\end{remark}

In \eqref{sc2d1}-\eqref{sc2d3}, $\hat{p},$ $\hat{\bf u}$ and $\widehat{{\bf u}c}$ are the numerical fluxes. We use alternating fluxes for the diffusion terms, and for any $e\in\Gamma_0$
\begin{equation}\label{flux2d1}
\hat{\bf u}|_e={\bf u}^+|_e,\quad\hat{p}|_e=p^-|_e,
\end{equation}
and on $\partial\Omega$, we take
$$
\hat{\bf u}|_e=0,\quad\hat{p}|_e=p^-|_e,\ \forall\ e\in\partial\Omega_+,\quad\hat{\bf u}|_e=0,\quad\hat{p}|_e=p^+|_e,\ \forall\ e\in\partial\Omega_-.
$$
For the convection term, we use
\begin{equation}\label{flux2d2}
\widehat{{\bf u}c}={\bf u}^+c^+-\alpha[c],
\end{equation}
Here $\alpha$ and $\tilde{\alpha}$ are two positive constants to be chosen by the bound-preserving technique.

Before we complete this subsection, we would like to introduce the following definition that will be used in the bound-preserving technique.
\begin{definition}\label{def}
We say two fluxes $\widehat{uc}$, $\hat{u}$ are consistent if $\widehat{uc}=\hat{u}$ by taking $c=1$ in $\Omega.$
\end{definition}
The numerical flux pair $(\widehat{uc}, \hat{u})$ given in \eqref{flux2d1} and \eqref{flux2d2} are consistent, and this is required by the bound-preserving technique.
\begin{remark}\label{remark}
There are plenty of flux pairs can be used following the procedures introduced in the next subsection. The proofs are basically the same with some minor changes, so we only list some of them below without more details.
\begin{itemize}
\item $\hat{u}=u^-$, $\hat{p}=p^+$, $\widehat{uc}=u^-c^--\alpha[c].$
\item $\hat{u}=\frac12(u^++u^-)$, $\hat{p}=\frac12(p^++p^-)$, $\widehat{uc}=\frac12(u^+c^++u^-c^-)-\alpha[c].$
\end{itemize}
\end{remark}

\section{Bound-preserving technique}
\label{secbb}
\subsection{One space dimension}
\label{secbb1d}
\setcounter{equation}{0}
In this subsection, we consider Euler forward time discretization and apply the bound-preserving technique to construct physically relevant numerical approximations in one space dimension.

%\subsection{The DG scheme}
%\label{secldg1d}
%We consider the one-dimensional version of \eqref{origin} on the spatial domain $\Omega=[0,2\pi]$ and solve
%\begin{equation*}
%\begin{array}{ll}
%d(c)p_t+u_x=d(c)p_t-(\frac{k(x)}{\mu(c)}p_x)_x=q,&x\in\Omega,\quad0< t\le T,\\
%\phi c_t+b(c) p_t+uc_x-(D(u)c_x)_x=(\tilde{c}-c)q,& x\in\Omega,\quad0< t\le T,
%\end{array}
%\end{equation*}
%with Neuamnn boundary conditions
%\begin{equation}\label{boundary1d}
%u=0,\quad D(u)c_x-cu=0\quad\textrm{on}\quad\partial\Omega.
%\end{equation}

%To construct the DG scheme, we divide the computational domain $\Omega$ into $N$ cells
%$$
%0=x_{\frac{1}{2}}<x_{\frac{3}{2}}<\cdots<x_{N+\frac{1}{2}}=2\pi,
%$$
%and denote
%$$
%I_j=\left(x_{j-\frac{1}{2}},x_{j+\frac{1}{2}}\right), \quad j=1,\cdots, N
%$$
%as the cells. For simplicity, we consider uniform meshes in this paper, and denote by $\Delta x$ the size of each cell. However, this assumption is not %essential. Following \cite{yifan}, we consider second-order DG scheme only, and
Define the finite element space to be
$$
V_h=\left\{v\in L^2(\Omega):v|_{I_j}\in\mathcal {P}^1(I_j),j=1,\cdots,N\right\},
$$
where $I_j$, $j=1,\cdots,N$ is a partition of the computational domain $\Omega=[0,2\pi]$.
%as the finite element space, where $\mathcal {P}^1(I_j)$ denotes the space of linear polynomials in $I_j$. Moreover, for any $v\in V_h$, we use %$v_{j+\frac12}^-=v(x^-_{j+\frac12})$ to denote the left limit of $v$ at $x_{j+\frac12}$. Likewise for $v^+$. Then the jump and average of $v$ at %$x_{j+\frac12}$ are given as $[v]_{j+\frac12}=v^+_{j+\frac12}-v^-_{j-\frac12}$ and $\{v\}_{j+\frac12}=\frac12(v^+_{j+\frac12}+v^-_{j-\frac12})$, %respectively. For simplicity, we define $v_{\frac12}^-=v_{N+\frac12}^+=0$.

%Now, we are ready to construct the DG scheme. For simplicity, if not otherwise stated, we use $p,u,c,r$ as the numerical approximations from now on. then
Then DG scheme is the following: find $p, u, r\in V_h$ such that for any $\xi,\eta,\zeta\in V_h$
{\setlength\arraycolsep{2pt}\begin{eqnarray}
(\tilde{d}(r)p_t,\xi)&=&(u,\xi_x)+\sum_{j=1}^{N-1}\hat{u}[\xi]_{j+\frac12}+(q,\xi)_j,\label{sc1d1}\\
(a(c)u,\eta)&=&(p,\eta_x)+\sum_{j=0}^N\hat{p}[\eta]_{j+\frac12},\label{sc1d2}\\
(r_t,\zeta)&=&\left(uc-D(u)c_x,\zeta_x\right)+\sum_{j=1}^{N-1}\widehat{uc}[\zeta]_{j+\frac12}-\sum_{j=1}^{N-1}\{D(u)c_x\}[\zeta]_{j+\frac12}\nonumber\\
&&-\sum_{j=1}^{N-1}\{D(u)\zeta_x\}[c]_{j+\frac12}-\sum_{j=1}^{N-1}\frac{\tilde{\alpha}}{h}[c][\zeta]_{j+\frac12}+(\tilde{c}q-rz_1p_t,\zeta),\label{sc1d3}
\end{eqnarray}}
where $c=\mathcal{I}_1\left\{\frac{r}{\Phi}\right\}$, $\tilde{d}(r)=z_1r+z_2(\Phi(x)-r)$ and $(u,v)=\int_\Omega uv dx$.% Here we take $c$ as the piecewise %linear approximation approximation of $\frac{r}{\Phi}$ such that in each cell $I_j$, we have
%$$
%c^+_{j-\frac12}=\frac{r^+_{j-\frac12}}{\Phi_{j-\frac12}},\quad c^-_{j+\frac12}=\frac{r^-_{j+\frac12}}{\Phi_{j+\frac12}}.
%$$

%In \eqref{sc1d1}-\eqref{sc1d3}, $\hat{u}$, $\hat{p}$ and $\widehat{uc}$ are the numerical fluxes. We use alternating fluxes for the diffusion term
%\begin{equation}\label{flux1}
%\hat{u}_{j+\frac12}=u^+_{j+\frac12},\quad\hat{p}_{j+\frac12}=p^-_{j+\frac12},\quad j=1,\cdots,N-1,
%\end{equation}
%and at the boundary, we choose
%$$
%\hat{p}_{\frac12}=p^+_{\frac12},\quad\hat{p}_{N+\frac12}=p^-_{N+\frac12}.
%$$
%For the convection term, we use the following flux
%\begin{equation}\label{flux2}
%\widehat{uc}=u^+c^+-\alpha[c],
%\end{equation}

For simplicity, we use $o_j$ for the numerical approximation $o$ in $I_j$, and the cell average is $\bar{o}_j$ for $o=u,p,c,r$. Moreover, we use $o^n$ to represent the solution $o$ at time level $n$. Take $\zeta=1$ in $I_j$ in \eqref{sc1d3} for $j=2,\cdots,N-1$ to obtain the equation satisfied by $\bar{r},$
\begin{equation}\label{cell_average}
\bar{r}_j^{n+1}=H_j^c(r,u,c)+H_j^d(r,u,c)+H^s_j(r,\tilde{c},q,z_1p_t),
\end{equation}
where
{\setlength\arraycolsep{2pt}\begin{eqnarray*}
H_j^c(r,u,c)&=&\frac13\bar{r}_j^n+\lambda\left(\widehat{uc}_{j-\frac12}-\widehat{uc}_{j+\frac12}\right),\\
H_j^d(r,u,c)&=&\frac13\bar{r}_j^n-\lambda\left(\{D(u)c_x\}_{j-\frac12}+\frac{\tilde{\alpha}}{\Delta x}[c]_{j-\frac12}-
\{D(u)c_x\}_{j+\frac12}-\frac{\tilde{\alpha}}{\Delta x}[c]_{j+\frac12}\right),\\
H^s_j(r,\tilde{c},q,z_1p_t)&=&\frac13\bar{r}_j^n+\Delta t\overline{\tilde{c}q-rz_1p_t},
\end{eqnarray*}}
with $\lambda=\frac{\Delta t}{\Delta x}$ being the ratio of the time and space mesh sizes, $\tilde{\alpha}$ being another parameter that will be chosen by the positivity-preserving technique, and $\overline{\tilde{c}q-rz_1p_t}$ being the cell average of $\tilde{c}q-rz_1p_t$. We will prove that if $\Delta t$ is sufficiently small, then $H_j^c$, $H_j^d$ and $H_j^s$ are all positive, and the results are given in the following three lemmas. For simplicity of presentation, if the denominator in a fraction is zero, then the value of the fraction is defined as $\infty$. Let us consider $H_j^c$ first.
\begin{lemma}\label{lemma1dc}
Suppose $r^n>0$ $(c^n>0)$, then $H_j^c(r,u,c)>0$ if $\alpha$ and $\lambda$ satisfy
\begin{equation}\label{cc}
\alpha>\max_{1\leq j\leq N-1}\left\{u_{j+\frac12}^+,0\right\},\quad\lambda\leq\min_{2\leq j\leq N-1}\left\{\frac{\Phi_m^0}{6\alpha},\frac{\Phi_{j-\frac12}}{6\left(\alpha-u^+_{j-\frac12}\right)}\right\},\quad\Phi_m^0=\min_{2\leq j\leq N-1}\Phi_{j+\frac12}.
\end{equation}
\end{lemma}
\begin{proof}
It is easy to check
{\setlength\arraycolsep{2pt}\begin{eqnarray*}
H_j^c(r,u,c)&=&\frac13\bar{r}_j^n+\lambda\left(\widehat{uc}_{j-\frac12}-\widehat{uc}_{j+\frac12}\right)\\
&=&\frac16\left(c^+\Phi_{j-\frac12}+c^-\Phi_{j+\frac12}\right)+\lambda\left(u^+c^+_{j-\frac12}-\alpha[c]_{j-\frac12}-u^+c^+_{j+\frac12}+\alpha[c]_{j+\frac12}\right)\\
&=&\alpha\lambda c^-_{j-\frac12}+\left(\frac16\Phi_{j-\frac12}+\lambda u^+_{j-\frac12}-\alpha\lambda\right)c^+_{j-\frac12}+\left(\frac16\Phi_{j+\frac12}-\alpha\lambda\right)c^-_{j+\frac12}+\lambda\left(\alpha-u^+_{j+\frac12}\right)c^+_{j+\frac12}.
\end{eqnarray*}}
Therefore, $H_j^c(r,u,c)>0$, if $\alpha$ and $\lambda$ satisfy condition \eqref{cc}.
\end{proof}

Now, we proceed to analyze $H_j^d$. Following the same analysis in \cite{yifan} with some minor changes, we can show the following lemma.
\begin{lemma}\label{lemma1dd}
Suppose $r^n>0$ $(c^n>0)$, then $H_j^d(r,u,c)>0$ under the condition
\begin{equation}\label{cd}
\tilde{\alpha}>\frac12D_M^0,\quad\Lambda=\frac{\Delta t}{\Delta x^2}\leq\min_{2\leq j\leq N-1}\frac{\Phi_{j\pm\frac12}}{3D_{j\pm\frac12}^\mp+6\tilde{\alpha}},\quad D_M^0=\max_{2\leq j\leq N-1}D^\pm_{j\pm\frac12}.
\end{equation}
\end{lemma}
\begin{proof}
Notice that $c$ is a linear function in each cell, then
$$
{c_x}_{j-\frac12}^-=\frac{c^-_{j-\frac12}-c^+_{j-\frac32}}{\Delta x},\quad
{c_x}_{j-\frac12}^+=\frac{c^-_{j+\frac12}-c^+_{j-\frac12}}{\Delta x}={c_x}_{j+\frac12}^-,\quad
{c_x}_{j+\frac12}^+=\frac{c^-_{j+\frac32}-c^+_{j+\frac12}}{\Delta x}.
$$
Therefore,
{\setlength\arraycolsep{2pt}\begin{eqnarray*}
H_j^d(r,u,s)&=&\frac13\bar{r}_j^n-\lambda\left(\{D(u)c_x\}_{j-\frac12}+\frac{\tilde{\alpha}}{\Delta x}[c]_{j-\frac12}-
\{D(u)c_x\}_{j+\frac12}-\frac{\tilde{\alpha}}{\Delta x}[c]_{j+\frac12}\right),\\
&=&\frac16\Phi_{j-\frac12}c^+_{j-\frac12}-\frac12\Lambda D^-_{j-\frac12}\left(c^-_{j-\frac12}-c^+_{j-\frac32}\right)-\frac12\Lambda D^+_{j-\frac12}\left(c^-_{j+\frac12}-c^+_{j-\frac12}\right)-\Lambda\tilde{\alpha}\left(c^+_{j-\frac12}-c^-_{j-\frac12}\right)\\
&&+\frac16\Phi_{j+\frac12}c^-_{j+\frac12}+\frac12\Lambda D^-_{j+\frac12}\left(c^-_{j+\frac12}-c^+_{j-\frac12}\right)+\frac12\Lambda D^+_{j+\frac12}\left(c^-_{j+\frac32}-c^+_{j+\frac12}\right)+\Lambda\tilde{\alpha}\left(c^+_{j+\frac12}-c^-_{j+\frac12}\right)\\
&=&\frac12\Lambda D^-_{j-\frac12}c^+_{j-\frac32}+\Lambda\left(\tilde{\alpha}-\frac12D^-_{j-\frac12}\right)c^-_{j-\frac12}+\left(\frac16\Phi_{j-\frac12}+\frac12\Lambda D^+_{j-\frac12}-\frac12\Lambda D^-_{j+\frac12}-\Lambda\tilde{\alpha}\right)c^+_{j-\frac12}\\
&&+\left(\frac16\Phi_{j+\frac12}-\frac12\Lambda D^+_{j-\frac12}+\frac12\Lambda D^-_{j+\frac12}-\Lambda\tilde{\alpha}\right)c^-_{j+\frac12}+\Lambda\left(\tilde{\alpha}-\frac12D^+_{j+\frac12}\right)c^+_{j+\frac12}+\frac12\Lambda D^+_{j+\frac12}c^-_{j+\frac32},
\end{eqnarray*}}
where $D^\pm_{j-\frac12}=D(u^\pm_{j-\frac12})\geq0.$ Therefore, we have $H_j^d(r,u,s)>0$ if
\eqref{cd} is satisfied.
\end{proof}

Next, we proceed to study $H_j^s$. We use Gaussian quadrature with two points to approximate the cell average of the source. The quadrature points on $I_j$ are denoted as $x^i_j$, $i=1,2$. Also, we denote $w_i$ as the corresponding weights on the interval $[-\frac12,\frac12]$. Then we can state the result.
\begin{lemma}\label{lemma1ds}
Suppose $r^n>0$ $(c^n>0)$, then $H_j^s(r,\tilde{c},q,z_1p_t)>0$ under the conditions
\begin{equation}\label{cs}
\Delta t\leq\frac{1}{6z_1p_M},\quad\Delta\leq\min_{1\leq j\leq N}\frac{\min\{\Phi_{j-\frac12},\Phi_{j+\frac12}\}}{6\max\left\{-q(x^1_j),-q(x^2_j),0\right\}}.
\end{equation}
where
\begin{equation}\label{pt}
p_M=\max_{i,j}\left(p_t(x^i_j),0\right).
\end{equation}
\end{lemma}
\begin{proof} The cell average of the source is approximated by the Gaussian quadrature, then
$$
H_j^s(r,u,c)=\frac13\bar{r}_j^n+\Delta t\overline{\tilde{c}q-rz_1p_t}=\sum_{i=1}^2w_i(L^1_i+L^2_i),
$$
where
$$
L^1_i=\left(\frac16-\Delta tz_1p_t(x^i_j)\right)r_j(x^i_j),\quad L^2_i=\frac16r_j(x^i_j)+\Delta t\tilde{c}q(x^i_j).
$$
Clearly, we have $L^1_i\geq 0$ if \eqref{cs} is satisfied. Moreover, $L^2_i>0$ if $q(x^i_j)\geq0$. We only need to consider the case with $q(x^i_j)<0$. Without loss of generality, we assume $q(x^1_j)<0$, then $\tilde{c}(x^1_j)=c(x^1_j)$. Notice that $r$ and $c$ are both linear functions in $I_j$ then it is easy to check
$$
r(x_j^1)=\mu_1r_{j-\frac12}^++\mu_2r_{j+\frac12}^-,\quad c(x_j^1)=\mu_1c_{j-\frac12}^++\mu_2c_{j+\frac12}^-,
$$
with $\mu_1=\frac{3+\sqrt{3}}{6}$ and $\mu_2=\frac{3-\sqrt{3}}{6}$. Therefore
{\setlength\arraycolsep{2pt}\begin{eqnarray*}
L^2_1&=&\frac16(\mu_1r_{j-\frac12}^++\mu_2r_{j+\frac12}^-)+\Delta t(\mu_1c_{j-\frac12}^++\mu_2c_{j+\frac12}^-)q(x^1_j)\\
&=&\left(\frac16\Phi_{j-\frac12}+\Delta tq(x^1_j)\right)\mu_1c_{j-\frac12}^++\left(\frac16\Phi_{j+\frac12}+\Delta tq(x^1_j)\right)\mu_2c_{j+\frac12}^-.
\end{eqnarray*}}
then $L^2_1>0$ under \eqref{cs}.
\end{proof}
Now we study the cells near $\partial\Omega$. For simplicity of presentation, we only demonstrate the results below, and the proof can be obtained exactly the same way in Lemmas \ref{lemma1dc}, \ref{lemma1dd} and \ref{lemma1ds} with some minor changes.
\begin{lemma}\label{lemma1db1}
Suppose $r^n_1>0$ $(c^n_1>0)$, then $\bar{r}_1^{n+1}>0$ under conditions \eqref{cs} and
$$
\alpha>u_{\frac32}^+,\quad\lambda\leq\frac{\Phi_\frac32}{6\alpha},\quad\textrm{and}\quad
\tilde{\alpha}>\frac12D_{\frac32}^+,\quad\Lambda\leq\min\left\{\frac{\Phi_\frac12}{3D_{\frac32}^-},\frac{\Phi_{\frac32}}{6\tilde{\alpha}}\right\}.
$$
Similarly, suppose $r^n_N>0$ $(c^n_N>0)$, then $\bar{r}_N^{n+1}>0$ under conditions \eqref{cs} and
$$
\lambda\leq\frac{\Phi_{N-\frac12}}{6\left(\alpha-u^+_{N-\frac12}\right)},\textrm{\quad and\quad}\tilde{\alpha}>\frac12D_{N-\frac12}^-,\quad\Lambda\leq\min\left\{\frac{\Phi_{N-\frac12}}{6\tilde{\alpha}},\frac{\Phi_{N+\frac12}}{3D_{N-\frac12}^+}\right\}.
$$
\end{lemma}
Based on the above lemmas, we can state the following theorem.
\begin{theorem}\label{thm1d1}
Consider the DG scheme \eqref{sc1d1}-\eqref{sc1d3} with Euler forward time discretization. Suppose $r^n>0$ $(c^n>0)$, and the parameters $\alpha$ and $\tilde{\alpha}$ are taken to be
$$
\alpha>\max_{1\leq j\leq N-1}\left\{u_{j+\frac12}^+,0\right\},\quad\tilde{\alpha}>\frac12D_M,\quad D_M=\max_{1\leq j\leq N-1}\left\{D^\pm_{j+\frac12}\right\}.
$$
Then $\bar{r}^{n+1}_j>0$ under \eqref{cs} and
$$
\lambda\leq\min_{1\leq j\leq N-1}\left\{\frac{\Phi_m}{6\alpha},\frac{\Phi_{j+\frac12}}{6\left(\alpha-u^+_{j+\frac12}\right)}\right\},\quad\Lambda\leq\min_{1\leq j\leq N}\frac{\Phi_{j\pm\frac12}}{6\tilde{\alpha}+3D^\mp_{j\pm\frac12}},\quad \Phi_m=\min_{1\leq j\leq N-1}\Phi_{j+\frac12}.
$$
\end{theorem}
\begin{proof}
For $j=0$ and $N$, the results was proved in Lemma \ref{lemma1db1}. For $j=1,2,\cdots,N-1$, by \eqref{cell_average}, we only need $H_j^c$, $H_j^d$ and $H_j^s$ to be positive, which follow directly from Lemmas \ref{lemma1dc}-\ref{lemma1ds}, respectively.
\end{proof}

The above theorem guarantees the positivity of $\bar{r}^{n+1}$. However, we still need to show $\bar{r}^{n+1}\leq\bar{\Phi},$ and the result is given below.
\begin{theorem}\label{thm1d2}
Suppose the conditions in Theorem \ref{thm1d1} are satisfied. Moreover, we assume $0\leq r^n\leq \Phi$ and the flux pair $(\widehat{uc}, \hat{u})$ are consistent, then $0\leq\bar{r}^{n+1}\leq\bar{\Phi},$ under another condition
$$
\Delta t\leq\frac{1}{6z_2p_M},
$$
where $p_M$ was given in \eqref{pt}.
\end{theorem}
\begin{proof}
Since the fluxes $\widehat{uc}$ and $\hat{u}$ are consistent, then $\widehat{uc}+\widehat{uc_2}=\hat{u}$, where $c_2=1-c.$ Take $\xi=\zeta$ in \eqref{sc1d1}, and subtract \eqref{sc1d3} from \eqref{sc1d1} to obtain
{\setlength\arraycolsep{2pt}\begin{eqnarray}
({r_2}_t,\zeta)&=&\left(uc_2-D(u){c_2}_x,\zeta_x\right)+\sum_{j=1}^{N-1}\widehat{uc_2}[\zeta]_{j+\frac12}-\sum_{j=1}^{N-1}\{D(u){c_2}_x\}[\zeta]_{j+\frac12}\nonumber\\
&&-\sum_{j=1}^{N-1}\{D(u)\zeta_x\}[c_2]_{j+\frac12}-\sum_{j=1}^{N-1}\frac{\tilde{\alpha}}{h}[c_2][\zeta]_{j+\frac12}+(\tilde{c}_2q-r_2z_2p_t,\zeta),\label{sc1d5}
\end{eqnarray}}
where $r_2=\Phi-r$ and $\tilde{c_2}=1-\tilde{c}.$ It is easy to check that \eqref{sc1d5} is exactly \eqref{sc1d3} with $r$, $c$ and $z_1$ replaced by $r_2,$ $c_2$ and $z_2$, respectively. Following the same analyses in Theorem \ref{thm1d1}, we can show that $\bar{r}^{n+1}_2>0$ under the conditions given in this theorem, which further implies $\bar{r}^{n+1}<\bar{\Phi}.$
\end{proof}
With Theorems \ref{thm1d2}, the numerical cell average $\bar{r}^n$ we constructed is between $0$ and $\bar{\Phi}$. However, the numerical approximation $r^n$ may be negative or larger than $\Phi$. Therefore, we also need to apply some limiter to $r^n$ and the procedure is given below. For simplicity, we drop the superscript $n$ in all the numerical approximations.
\begin{enumerate}
\item Set up a small number $\epsilon=10^{-13}.$
\item If $\bar{r}\leq\epsilon$, take $r=\bar{r}$. If $\bar{r}_2\leq\epsilon$, take $r=\Phi-\bar{r}_2$. Then skip the following steps.
\item\label{enumerate} Define $m_j=\min\{r_{j-\frac12}^+,r_{j+\frac12}^-\}$. If $m_j\geq0$ then proceed to the next step. Otherwise, without loss of generality, we assume $r_{j-\frac12}^+<0$, then define
$$
\tilde{r}_{j-\frac12}^+=\epsilon,\quad\tilde{r}_{j+\frac12}^-=r_{j+\frac12}^--\epsilon+r_{j-\frac12}^+
$$
and use $\tilde{r}_{j-\frac12}^+$ and $\tilde{r}_{j+\frac12}^-$ to construct a new approximation, also denoted as $r$.
\item Apply the previous step for $r_2=\Phi-r$ and construct a new approximation $r$.
\end{enumerate}
\begin{remark}
The above algorithm keeps the second-order accuracy. In step \ref{enumerate}, we assume $r_{j-\frac12}^+<0$ and $0\leq\Phi_{j-\frac12}-r_{j-\frac12}^+\leq Ch^2$, then it is very easy to check $0\leq\Phi_{j-\frac12}-\tilde{r}_{j-\frac12}^+\leq Ch^2$.
\end{remark}

\begin{remark}
For an $n$-component fluid $(n\geq3)$, we denote the concentration of the $i$'th component to be $c_i$, and let $r_i=\phi c_i.$ Follow the same analysis in Theorem \ref{thm1d2}, we can show $\bar{r_i}\geq 0, i=1,\cdots,n$ and $\sum_{i=1}^n \bar{r}_i=\bar{\Phi}$. However, the limiter is not easy to construct. We will work on this in the future.
\end{remark}

\begin{remark}
The extension to high-order schemes is not straightforward as demonstrated in \cite{yifan}. Recently, Chen \cite{Chen2016} introduced the third-order maximum-principle preserving direct DG methods for convection-diffusion equations which can be applied to this problem. Another approach is to applied the flux limiter \cite{Xu2014,Xiong2015}. However, the accuracy was demonstrated by numerical experiments only.
\end{remark}

\subsection{High order time discretizations}
\label{sectime}
All the previous analyses are based on first-order Euler forward time discretization. We can also use strong stability preserving (SSP) high-order time discretizations to solve the ODE system ${\bf w}_t={\bf L}{\bf w}.$ More details of these time discretizations can be found in \cite{SO,time2,time1}. In this paper, we use the third-order SSP Runge-Kutta method \cite{SO}
{\setlength\arraycolsep{2pt}\begin{eqnarray}
{\bf w}^{(1)}&=&{\bf w}^n+\Delta t{\bf L}({\bf w}^n),\nonumber\\
{\bf w}^{(2)}&=&\frac34{\bf w}^n+\frac14\left({\bf w}^{(1)}+\Delta
t{\bf L}({\bf w}^{(1)})\right),\\
{\bf w}^{n+1}&=&\frac13{\bf w}^n+\frac23\left({\bf w}^{(2)}+\Delta
t{\bf L}({\bf w}^{(2)})\right),\nonumber
\end{eqnarray}}
and the third order SSP multi-step method \cite{time2}
\begin{equation}
{\bf w}^{n+1}=\frac{16}{27}\left({\bf w}^n+3\Delta t{\bf L}({\bf
w}^n)\right)+\frac{11}{27}\left({\bf w}^{n-3}+\frac{12}{11}\Delta
t{\bf L}({\bf w}^{n-3})\right).
\end{equation}
Since an SSP time discretization is a convex combination of Euler forward and the algorithm is to construct positive $r$ and $\Phi-r$, hence by using the limiter designed in section \ref{secbb1d}, the numerical solution obtained from the full scheme is also physically relevant.

\begin{remark}
We can also applied the second-order explicit SSP Runge-Kutta methods \cite{SO}. For implicit time integration, it is not easy to prove that the numerical cell averages are physically relevant as demonstrated in Theorems \ref{thm1d1} and \ref{thm1d2}, since the fluxes are coupled together. We will continue this in the future.
\end{remark}

\subsection{Two space dimensions}
\label{secbb2d}
In this section, we apply the bound-preserving technique to the DG scheme in two space dimensions. We consider Euler-forward time discretization only, and the high-order ones were discussed in Section \ref{sectime}. For simplicity of presentation, we only discuss the techniques for cells away from $\partial\Omega$, while the boundary cells can be analyzed following the same lines. We use $o_{ij}$ for the numerical approximation $o$ in $K_{ij}$ and the cell average is $\bar{o}_{ij}$. Moreover, we use $o^n$ to represent the solution $o$ at time level $n$.

In \eqref{sc2d3}, we take $\zeta=1$ in $K_{ij}$ to obtain the equation satisfied by the cell average of $r$,
\begin{equation}\label{fluxintegral}
\bar{r}^{n+1}_{ij}=H^c(r,{\bf u},c)+H^d_x(r,{\bf u},c)+H^d_y(r,{\bf u},c)+H^s(r,\tilde{c},q,z_1p_t),
\end{equation}
where
{\setlength\arraycolsep{2pt}\begin{eqnarray*}
H^c(r,{\bf u},c)&=&\frac13\bar{r}^n_{ij}+\lambda\left(\int_{J_j}\widehat{u_1c}_{i-\frac12,j}-\widehat{u_1c}_{i+\frac12,j}dy
+\int_{I_i}\widehat{u_2c}_{i,j-\frac12}-\widehat{u_2c}_{i,j+\frac12}dx\right),\\
H^d_x(r,{\bf u},c)&=&\frac16\bar{r}^n_{ij}-\lambda\left(\int_{J_j}\left\{D_{11}c_x+D_{12}c_y\right\}_{i-\frac12,j}+\frac{\tilde{\alpha}}{|J_j|}[c]_{i-\frac12,j}\right.\\
&&\left.-\left\{D_{11}c_x+D_{12}c_y\right\}_{i+\frac12,j}-\frac{\tilde{\alpha}}{|J_j|}[c]_{i+\frac12,j}dy\right),\\
H^d_y(r,{\bf u},c)&=&\frac16\bar{r}^n_{ij}-\lambda\left(\int_{I_i}\left\{D_{21}c_x+D_{22}c_y\right\}_{i,j-\frac12}+\frac{\tilde{\alpha}}{|I_i|}[c]_{i,j-\frac12}\right.\\
&&\left.-\left\{D_{21}c_x+D_{22}c_y\right\}_{i,j+\frac12}-\frac{\tilde{\alpha}}{|I_i|}[c]_{i,j+\frac12}dy\right),\\
H^s(r,\tilde{c},q,z_1p_t)&=&\frac13\bar{r}^n_{ij}+\Delta t\overline{\tilde{c}q-rz_1p_t},\\
\end{eqnarray*}}
with $\lambda=\frac{\Delta t}{\Delta x\Delta y}$ and ${\bf u}=(u_1,u_2)^T.$
To continue, we use 2-point Gaussian quadratures to approximate the integrals given above. The Gaussian quadrature points on $I_i$ and $J_j$ are denoted by
$
\left\{x_i^1,x_i^2\right\}\quad\textrm{and}\quad\left\{y_j^1,y_j^2\right\},
$
respectively. The corresponding weights on the interval $[-\frac12,\frac12]$ are denoted as $w_1$ and $w_2$. Denote $\lambda_1=\frac{\Delta t}{\Delta x}$ and $\lambda_2=\frac{\Delta t}{\Delta y}$, then
$$
H^c(r,{\bf u},c)=\frac13\bar{r}^n_{ij}+\lambda_1\sum_{\beta=1}^2w_\beta\left[\widehat{u_1c}_{i-\frac12,j,\beta}-\widehat{u_1c}_{i+\frac12,j,\beta}\right]\
+\lambda_2\sum_{\beta=1}^2w_\beta\left[\widehat{u_2c}_{i,j-\frac12,\beta}-\widehat{u_2c}_{i,j+\frac12,\beta}\right],
$$
where $\widehat{u_1c}_{i-\frac12,j,\beta}=\widehat{u_1c}_{i-\frac12,j}(y_j^\beta)$ is a point value in the Gaussian quadrature on the edge $e=\{x_{i-\frac12}\}\times J_j$. Likewise for the other
point values. As the general treatment, we rewrite the cell average on the right hand side as
$$
\bar{r}^n_{ij}=\sum_{\beta=1}^2\frac{w_\beta}2\left(r^+_{i-\frac12,j,\beta}+r^-_{i+\frac12,j,\beta}\right)
=\sum_{\beta=1}^2\frac{w_\beta}2\left(r^+_{i,j-\frac12,\beta}+r^-_{i,j+\frac12,\beta}\right).
$$
Based on the above,
{\setlength\arraycolsep{2pt}\begin{eqnarray*}
H^c(r,{\bf u},c)&=&\sum_{\beta=1}^2w_\beta\left[\frac{\lambda_1}{6(\lambda_1+\lambda_2)}\left(r^+_{i-\frac12,j,\beta}+r^-_{i+\frac12,j,\beta}\right)
+\lambda_1\left(\widehat{u_1c}_{i-\frac12,j,\beta}-\widehat{u_1c}_{i+\frac12,j,\beta}\right)\right]\\
&+&\sum_{\beta=1}^2w_\beta\left[\frac{\lambda_2}{6(\lambda_1+\lambda_2)}\left(r^+_{i,j-\frac12,\beta}+r^-_{i,j+\frac12,\beta}\right)
+\lambda_2\left(\widehat{u_2c}_{i,j-\frac12,\beta}-\widehat{u_2c}_{i,j+\frac12,\beta}\right)\right].
\end{eqnarray*}}
Following the same proof in Lemma \ref{lemma1dc} with some minor changes, we have
\begin{lemma}\label{lemma2dc}
Suppose $r^n>0$ $(c^n>0)$ and the parameter $\alpha$ is taken to be
\begin{equation}\label{alpha}
\alpha>\max_{\arraycolsep=0pt\def\arraystretch{0.8}\begin{array}{c}\text{\footnotesize$1\leq i\leq N_x-1,$}\\\text{\footnotesize$1\leq j\leq N_y-1,$}\\\text{\footnotesize$\beta=1,2$}\end{array}}\{{u_1}_{i+\frac12,j,\beta}^+,{u_2}_{i,j+\frac12,\beta}^+,0\},
\end{equation}
then $H^c(r,{\bf u},c)>0$ under the condition
\begin{equation}\label{cc2d}
\frac{\Delta t}{\Delta x}+\frac{\Delta t}{\Delta y}\leq\frac16\min\left\{\frac{\Phi_m}{\alpha},B_1,B_2\right\}
\end{equation}
where
$$
B_1=\min_{\arraycolsep=0pt\def\arraystretch{0.8}\begin{array}{c}\text{\footnotesize$1\leq i\leq N_x-1$}\\\text{\footnotesize$1\leq j\leq N_y$}\\\text{\footnotesize$\beta=1,2$}\end{array}}\frac{\Phi_{i+\frac12,j\pm\frac12}}{\alpha-u^+_{i+\frac12,j,\beta}},
\quad
B_2=\min_{\arraycolsep=0pt\def\arraystretch{0.8}\begin{array}{c}\text{\footnotesize$1\leq i\leq N_x$}\\\text{\footnotesize$1\leq j\leq N_y-1$}\\\text{\footnotesize$\beta=1,2$}\end{array}}\frac{\Phi_{i\pm\frac12,j+\frac12}}{\alpha-u^+_{i,j+\frac12,\beta}},
\quad
\Phi_m=\min_{\arraycolsep=0pt\def\arraystretch{0.8}\begin{array}{c}\text{\footnotesize$0\leq i\leq N_x$}\\\text{\footnotesize$0\leq j\leq N_y$}\end{array}}\Phi_{i+\frac12,j+\frac12}.
$$
\end{lemma}
For $H^d_x$ and $H^d_y$, the analyses are similar.
\begin{lemma}\label{lemma2dd}
Suppose $r^n>0$ $(c^n>0)$, then $H_x^d(r,{\bf u},c)>0$ under the conditions
\begin{equation}\label{alphatilde1}
\tilde{\alpha}\geq \frac{\Delta y}{2\Delta x}D_{11}^M+\sqrt{3} D_{12}^M,
\end{equation}
\begin{equation}\label{cd2dx}
D_{11}^M\Lambda_1+2(\tilde{\alpha}+D_{12}^M)\lambda\leq\frac1{12}\Phi_m,
\end{equation}
where
$$
D_{11}^M=\max_{(x,y)\in\Omega} D_{11}({\bf u})(x,y),\quad D_{12}^M=\max_{(x,y)\in\Omega} |D_{12}({\bf u})(x,y)|,\quad\Lambda_1=\frac{\Delta t}{\Delta x^2}.
$$
Similarly, we have $H_y^d(r,{\bf u},c)>0$ if
\begin{equation}\label{alphatilde2}
\tilde{\alpha}\geq \frac{\Delta x}{2\Delta y}D_{22}^M+\sqrt{3} D_{21}^M,
\end{equation}
\begin{equation}\label{cd2dy}
D_{22}^M\Lambda_2+2(\tilde{\alpha}+D_{21}^M)\lambda\leq\frac1{12}\Phi_m,
\end{equation}
where
$$
D_{22}^M=\max_{(x,y)\in\Omega} D_{22}({\bf u})(x,y),\quad D_{21}^M=\max_{(x,y)\in\Omega} |D_{21}({\bf u})(x,y)|,\quad\Lambda_2=\frac{\Delta t}{\Delta y^2}.
$$
\end{lemma}
\begin{proof}
We only prove for $H^d_x$. Moreover, we simplify the subscript of the point value in a Gaussian quadrature, and denote $c^+_{i-\frac12,\beta}$ for $c^+_{i-\frac12,j,\beta}.$ Likewise for the others. Notice the fact that, $c$ is linear along each axis. Therefore,
$$
{c_x}_{i-\frac12,\beta}^-=\frac1{\Delta x}\left(c_{i-\frac12,\beta}^--c_{i-\frac32,\beta}^+\right),\quad
{c_x}_{i-\frac12,\beta}^+=\frac1{\Delta x}\left(c_{i+\frac12,\beta}^--c_{i-\frac12,\beta}^+\right),
$$
$$
{c_y}_{i-\frac12,1}^-=\frac{\sqrt{3}}{\Delta y}\left(c_{i-\frac12,2}^--c_{i-\frac12,1}^-\right)={c_y}_{i-\frac12,2}^-,\quad
{c_y}_{i-\frac12,1}^+=\frac{\sqrt{3}}{\Delta y}\left(c_{i-\frac12,2}^+-c_{i-\frac12,1}^+\right)={c_y}_{i-\frac12,2}^+.
$$
We also use 2-point Gaussian quadratures to approximate the integral, then
{\setlength\arraycolsep{2pt}\begin{eqnarray*}
H^d_x(r,u,c)&=&\frac16\bar{r}^n_{ij}-\frac{\lambda_1}2\sum_{\beta=1}^2w_\beta\left(\frac{{D_{11}}^-_{i-\frac12,\beta}}{\Delta x}\left(c_{i-\frac12,\beta}^--c_{i-\frac32,\beta}^+\right)+\frac{{D_{11}}^+_{i-\frac12,\beta}}{\Delta x}\left(c_{i+\frac12,\beta}^--c_{i-\frac12,\beta}^+\right)\right)\\
&&-\frac{\sqrt{3}\lambda_1}2\sum_{\beta=1}^2w_\beta\left(\frac{{D_{12}}^-_{i-\frac12,\beta}}{\Delta y}\left(c_{i-\frac12,2}^--c_{i-\frac12,1}^-\right)+\frac{{D_{12}}^+_{i-\frac12,\beta}}{\Delta y}\left(c_{i-\frac12,2}^+-c_{i-\frac12,1}^+\right)\right)\\
&&+\frac{\lambda_1}2\sum_{\beta=1}^2w_\beta\left(\frac{{D_{11}}^-_{i+\frac12,\beta}}{\Delta x}\left(c_{i+\frac12,\beta}^--c_{i-\frac12,\beta}^+\right)+\frac{{D_{11}}^+_{i+\frac12,\beta}}{\Delta x}\left(c_{i+\frac32,\beta}^--c_{i+\frac12,\beta}^+\right)\right)\\
&&+\frac{\sqrt{3}\lambda_1}2\sum_{\beta=1}^2w_\beta\left(\frac{{D_{12}}^-_{i+\frac12,\beta}}{\Delta y}\left(c_{i+\frac12,2}^--c_{i+\frac12,1}^-\right)+\frac{{D_{12}}^+_{i+\frac12,\beta}}{\Delta y}\left(c_{i+\frac12,2}^+-c_{i+\frac12,1}^+\right)\right)\\
&&-\lambda_1\sum_{\beta=1}^2w_\beta\frac{\tilde{\alpha}}{\Delta y}\left(c^+_{i-\frac12,\beta}-c^-_{i-\frac12,\beta}\right)
+\lambda_1\sum_{\beta=1}^2w_\beta\frac{\tilde{\alpha}}{\Delta y}\left(c^+_{i+\frac12,\beta}-c^-_{i+\frac12,\beta}\right)\\
&=&\frac16\bar{r}^n_{ij}+\Lambda_1\sum_{\beta=1}^2\frac{w_\beta}2{D_{11}}_{i-\frac12,\beta}^-c_{i-\frac32,\beta}^+
+\Lambda_1\sum_{\beta=1}^2\frac{w_\beta}2{D_{11}}^+_{i+\frac12,\beta}c_{i+\frac32,\beta}^-\\
&&+\sum_{\beta=1}^2\frac{w_\beta}2\left(\tau_{i-\frac12,\beta}^-c_{i-\frac12,\beta}^-+\tau_{i-\frac12,\beta}^+c_{i-\frac12,\beta}^+
+\tau_{i+\frac12,\beta}^-c_{i+\frac12,\beta}^-+\tau_{i+\frac12,\beta}^+c_{i+\frac12,\beta}^+\right),
\end{eqnarray*}}
where
{\setlength\arraycolsep{2pt}\begin{eqnarray*}
\tau_{i\pm\frac12,1}^\pm&=&-\Lambda_1{D_{11}}^\pm_{i\pm\frac12,1}\mp\sqrt{3}\lambda{D_{12}}^\pm_{i\pm\frac12,1}\mp\sqrt{3}\lambda{D_{12}}^\pm_{i\pm\frac12,2}+2\lambda\tilde{\alpha},\\
\tau_{i\pm\frac12,2}^\pm&=&-\Lambda_1{D_{11}}^\pm_{i\pm\frac12,2}\pm\sqrt{3}\lambda{D_{12}}^\pm_{i\pm\frac12,1}\pm\sqrt{3}\lambda{D_{12}}^\pm_{i\pm\frac12,2}+2\lambda\tilde{\alpha},\\
\tau_{i\mp\frac12,1}^\pm&=&\pm\Lambda_1{D_{11}}^+_{i-\frac12,1}\pm\sqrt{3}\lambda{D_{12}}^\pm_{i\mp\frac12,1}\pm\sqrt{3}\lambda{D_{12}}^\pm_{i\mp\frac12,2}\mp\Lambda_1{D_{11}}^-_{i+\frac12,1}-2\lambda\tilde{\alpha},\\
\tau_{i\mp\frac12,2}^\pm&=&\pm\Lambda_1{D_{11}}^+_{i-\frac12,2}\mp\sqrt{3}\lambda{D_{12}}^\pm_{i\mp\frac12,1}\mp\sqrt{3}\lambda{D_{12}}^\pm_{i\mp\frac12,2}\mp\Lambda_1{D_{11}}^-_{i+\frac12,2}-2\lambda\tilde{\alpha}.
\end{eqnarray*}}
Since $D_{11}>0$, then we only need to consider the terms on the second line in the last step. Firstly, if we take
$$
\tilde{\alpha}\geq \frac{\Delta y}{2\Delta x}D_{11}^M+\sqrt{3} D_{12}^M,
$$
where
$$
D_{11}^M=\max_{(x,y)\in\Omega} D_{11}({\bf u})(x,y),\quad D_{12}^M=\max_{(x,y)\in\Omega} |D_{12}({\bf u})(x,y)|.
$$
then $\tau_{i-\frac12,\beta}^->0$ and $\tau_{i+\frac12,\beta}^+>0$. For all the other $\tau's$, they are related to the point values of $c$ in cell $K_{ij}$. It is easy to see
$$
c_{i\pm\frac12,1}=\mu_1c_{i\pm\frac12,j-\frac12}+\mu_2c_{i\pm\frac12,j+\frac12},\quad
c_{i\pm\frac12,2}=\mu_2c_{i\pm\frac12,j-\frac12}+\mu_1c_{i\pm\frac12,j+\frac12},
$$
where $\mu_1=\frac{3+\sqrt{3}}{6}$, $\mu_2=\frac{3-\sqrt{3}}{6}$ and the point values of $c$ mentioned above are all the ones chosen in cell $K_{ij}$. In this section, if not otherwise stated, this is the default definition of element in $W_h$ along the traces in $K_{ij}$. With the definition above, we can write
$$
\bar{r}^n_{ij}=\frac{(c\Phi)_{i-\frac12,j-\frac12}+(c\Phi)_{i-\frac12,j+\frac12}+(c\Phi)_{i+\frac12,j-\frac12}+(c\Phi)_{i+\frac12,j+\frac12}}{4},
$$
which further yields
{\setlength\arraycolsep{2pt}\begin{eqnarray*}
&&\frac16\bar{r}^n_{ij}+\sum_{\beta=1}^2\frac{w_\beta}2\tau_{i-\frac12,\beta}^+c_{i-\frac12,\beta}^++\sum_{\beta=1}^2\frac{w_\beta}2\tau_{i+\frac12,\beta}^-c_{i+\frac12,\beta}^-\\
&=&\frac12A_1c_{i-\frac12,j-\frac12}+\frac12A_2c_{i-\frac12,j+\frac12}+\frac12A_3c_{i+\frac12,j-\frac12}+\frac12A_4c_{i+\frac12,j+\frac12},
\end{eqnarray*}}
where
{\setlength\arraycolsep{2pt}\begin{eqnarray*}
A_1=\frac1{12}\Phi_{i-\frac12,j-\frac12}+w_1\mu_1\tau_{i-\frac12,1}^++w_2\mu_2\tau_{i-\frac12,2}^+,&\quad&
A_2=\frac1{12}\Phi_{i-\frac12,j+\frac12}+w_1\mu_2\tau_{i-\frac12,1}^++w_2\mu_1\tau_{i-\frac12,2}^+,\\
A_3=\frac1{12}\Phi_{i+\frac12,j-\frac12}+w_1\mu_1\tau_{i+\frac12,1}^-+w_2\mu_2\tau_{i+\frac12,2}^-,&\quad&
A_4=\frac1{12}\Phi_{i+\frac12,j+\frac12}+w_1\mu_2\tau_{i+\frac12,1}^-+w_2\mu_1\tau_{i+\frac12,2}^-.
\end{eqnarray*}}
We want all the $A_i's$ to be positive, and only prove for $A_1$, since the other three can be analyzed by the same lines. It is easy to check that
{\setlength\arraycolsep{2pt}\begin{eqnarray*}
A_1&=&\frac1{12}\Phi_{i-\frac12,j-\frac12}+\sum_{\beta=1}^2w_\beta\mu_\beta\left(\Lambda_1{D_{11}}^+_{i-\frac12,\beta}-\Lambda_1{D_{11}}^-_{i+\frac12,\beta}-2\lambda\tilde{\alpha}\right)\\
&&+(w_1\mu_1-w_2\mu_2)\sqrt{3}\lambda\left({D_{12}}^+_{i-\frac12,1}+{D_{12}}^+_{i-\frac12,2}\right)\\
&\geq&\frac1{12}\Phi_{i-\frac12,j-\frac12}-2\lambda\tilde{\alpha}-D^M_{11}\Lambda_1-2D_{12}^M\lambda.
\end{eqnarray*}}
Therefore, by taking
$$
D_{11}^M\Lambda_1+2(\tilde{\alpha}+D_{12}^M)\lambda\leq\frac1{12}\Phi_{i-\frac12,j-\frac12},
$$
we have
$A_1>0.$
\end{proof}
For $H^s$, the result is similar to Lemma \ref{lemma1ds}, so we skip the proof and state the result below.
\begin{lemma}\label{lemma2ds}
Suppose $c^n>0$ and $r^n>0$ then $H^s(r,{\bf u},c)>0$ under the condition
\begin{equation}\label{cs2d}
\Delta t\leq\frac{1}{6z_1p_M},\quad\textrm{and}\quad
\Delta t\leq\min_{i,j}\frac{\Phi^m_{ij}}
{6\max_{\beta,\gamma}\left\{-q(x^\beta_i,y^\gamma_j),0\right\}},
\end{equation}
where
\begin{equation}\label{pt2d}
p_M=\max_{\beta,\gamma=1,2}p_t(x_i^\beta,y_j^\gamma),\quad\Phi^m_{ij}=\min\{\Phi_{i-\frac12,j-\frac12},
\Phi_{i-\frac12,j+\frac12},\Phi_{i+\frac12,j-\frac12},\Phi_{i+\frac12,j+\frac12}\}.
\end{equation}
\end{lemma}
Based on the above three lemmas, we can state the following theorem.
\begin{theorem}\label{thm2d1}
Suppose $r^n>0$ $(c^n>0)$, and the parameters $\alpha$ and $\tilde{\alpha}$ satisfy \eqref{alpha} and \eqref{alphatilde1}, \eqref{alphatilde2}, respectively. Then $\bar{r}^{n+1}_j>0$ under the conditions \eqref{cc2d}, \eqref{cd2dx}, \eqref{cd2dy} and \eqref{cs2d}.
\end{theorem}
Following the same proof in Theorem \ref{thm1d2} with some minor changes, we have the following one.
\begin{theorem}\label{thm2d2}
Suppose $0\leq r^n\leq \Phi$ and the conditions in Theorem \ref{thm2d1} are satisfied. Moreover, if the fluxes $\widehat{{\bf u}c}$ and $\hat{\bf u}$ are consistent, then $\bar{r}^{n+1}\leq\bar{\Phi},$ under the condition
\begin{equation}\label{cs22d}
\Delta t\leq\frac{1}{6z_2p_M}
\end{equation}
where $p_M$ was given in \eqref{pt2d}.
\end{theorem}
With Theorem \ref{thm2d2}, we can construct physically relevant numerical cell averages $\bar{r}^n$. However, the numerical approximation $r^n$ may be negative or larger than $\Phi$. Therefore, we also need to apply limiters to $r^n$ and the procedure is given in the following steps. For simplicity, we drop the superscript $n$ in all the numerical approximations
\begin{enumerate}
\item Set up a small number $\epsilon=10^{-13}.$
\item If $\bar{r}<\epsilon$, we take $r=\bar{r}$. If $\bar{\Phi}-\bar{r}<\epsilon$, we take $r=\Phi-\bar{\Phi}+\bar{r}$. Then skip the following steps.
\item\label{enumerate2d} Compute $m_j=\textrm{min}_{(x,y)\in K_{ij}}r(x,y)$. If $m_j<0$, then define the four corners of $K_{ij}$ to be ${\bf x}_\ell$, $\ell=1,2,3,4$. Compute
$$
s=\frac{\bar{r}}{\sum_{\ell=1}^4(r({\bf x}_\ell)+|r({\bf x}_\ell)|)},
$$
then the values of the modified $r$ at the corners are given as
$$
\tilde{r}({\bf x}_i)=\left\{\begin{array}{cc}0,&\textrm{if}\quad r({\bf x}_i)<0,\\sr({\bf x}_i),&\textrm{if}\quad r({\bf x}_i)>0.\end{array}\right.
$$
Finally, we use $\tilde{r}({\bf x}_i)$ to construct a new approximation, also denoted as $r$.
\item Apply the previous step for $r_2=\Phi-r$ and construct a new approximation $r$.
\end{enumerate}

\begin{remark}
The algorithm above can be applied to triangular meshes, but the finite element space should be different. In general, we require $\Phi$, the projection of $\phi$, to be continuous. For rectangular meshes, there are four degrees of freedom (DOF) to be determined, and we have to use $Q^1$ polynomials. For triangular meshes, only three DOFs available, hence $P^1$ polynomials is necessary. Besides the above, the proof for triangular meshes is also different. Following \cite{yifan}, the integral of the fluxes in \eqref{fluxintegral} can be written as $\int_{\partial K}{\bf D}\nabla c\cdot {\bf n}ds=\int_{\partial K}\nabla c\cdot ({\bf Dn})ds=\int_{\partial K}\frac{\partial c}{\partial \boldsymbol\nu}ds$, where $K$ is a cell in the partition,  ${\bf n}$ is the outer normal of the boundary $\partial K$ and ${\boldsymbol\nu}={\bf Dn}$. For $P^1$ element, the directional derivative $\frac{\partial c}{\partial \boldsymbol\nu}$ is a constant along the direction $\boldsymbol\nu$ in each cell. Notice the fact that $c_x$ and $c_y$ are constants in this section. Therefore, we can follow the analysis in this paper to estimate $\frac{\partial c}{\partial \boldsymbol\nu}$ and obtain positive numerical cell averages $\bar{r}$.
\end{remark}

%\begin{remark}
%The bound-preserving technique can also be extended to compressible miscible displacements in three space dimensions following the same analysis above %with some minor changes. So we skip the details here.
%\end{remark}

\section{Numerical example}
\label{secexample}
\setcounter{equation}{0}
In this section, we provide numerical examples to illustrate the accuracy and capability of the method. The time discretization is given as the third-order SSP Runge-Kutta method \cite{time1}.

\subsection{One space dimension}
In this subsection, we consider problem in one space dimension on the computational domain $[0,2\pi]$. If not otherwise stated, we choose $N=80$. In the first example, we would like to construct an analytical solution and test the accuracy of the bound-preserving DG scheme.
\begin{example}\label{accuracy}
We choose the initial condition as
$$
c(x,0)=\frac12(1-\cos(x)),\quad\quad p(x,0)=\cos(x)-1.
$$
The source parameters $q$ and $\tilde{c}$ are taken as
$$
q(x,t)=e^{-t},\quad\tilde{c}=\frac12(e^{-\gamma t}(\sin^2(x)-\cos(x))+1).
$$
Moreover, we also choose other parameters as
$$
z_1=z_2=1,\quad\phi(x)=1,\quad D(u)=\gamma,\quad \kappa(x)=\mu(c)=1.
$$
\end{example}
It is easy to see that the exact solutions are
$$
c(x,t)=\frac12(1-e^{-\gamma t}\cos(x)),\quad\quad p(x,t)=e^{-t}(\cos(x)-1).
$$
We choose $\Delta t=0.05\Delta x^2$ and compute up to $T=1$. We take $\gamma=10^{-5}$ such that the exact solution $c$ is very close to $0$ for $t>0$. Therefore, the bound-preserving limiter enact frequently. We compute the $L^\infty$-norm of the error between the numerical and exact solutions of $c$, and the results are given in Table \ref{accuracy_test}.
\begin{table}[!htb]
\centering
\begin{tabular}{c|c|c|c|c}
\hline
&\multicolumn{2}{c|}{with limiter}&\multicolumn{2}{c}{ no limiter}\\\hline
$N$ & $L^{\infty}$ error & order&$L^{\infty}$ error & order
\\\hline
20 &4.02e-3& -- &3.21e-3& -- \\%30.96%
40 &1.02e-3&1.98&8.15e-4&1.98\\%15.31%
80 &2.57e-4&1.99&2.07e-4&2.00\\%7.60%
160&6.41e-5&2.00&5.07e-5&2.00\\%3.78%
\hline
\end{tabular}
\caption{Example \ref{accuracy}: accuracy test for the second-order DG methods with and without the bound-preserving limiter.}
\label{accuracy_test}
\end{table}
From the table, we can observe second-order accuracy of the DG scheme with and without the bound-preserving limiter. Therefore, the bound-preserving technique does not kill the accuracy. This is quite different from the TVD-like schemes, which has only first order accuracy in $L^\infty$-norm. Next, we test the effect of the bound-preserving limiter and solve the following example.
\begin{example}\label{ex2}
We chose the initial condition as
$$
c(x,0)=\frac12(\cos(x)+1),\quad\quad p(x,0)=-\gamma\cos(x).
$$
Other parameters are chosen as
$$
q(x,t)=0,\ z_1=0.35,\ z_2=\mu(c)=\kappa(x)=1,\ \phi(x)=\frac14(3+\cos(x)),\ D(u)=0.
$$
\end{example}
In this example, we use $\Delta t=0.01\Delta x^2$ and compute up to $T=0.1$. Since we take $D(u)=0$, the diffusion term cannot provide any stability to the numerical scheme. In the numerical experiments, we consider $\gamma=1,2,3,4$ and apply the bound-preserving limiters. The numerical approximations of $c$ are given in the left panel of Figure \ref{ex2_figure}.
\begin{figure}[!htp]\centering
    \includegraphics[width=0.35\textwidth]{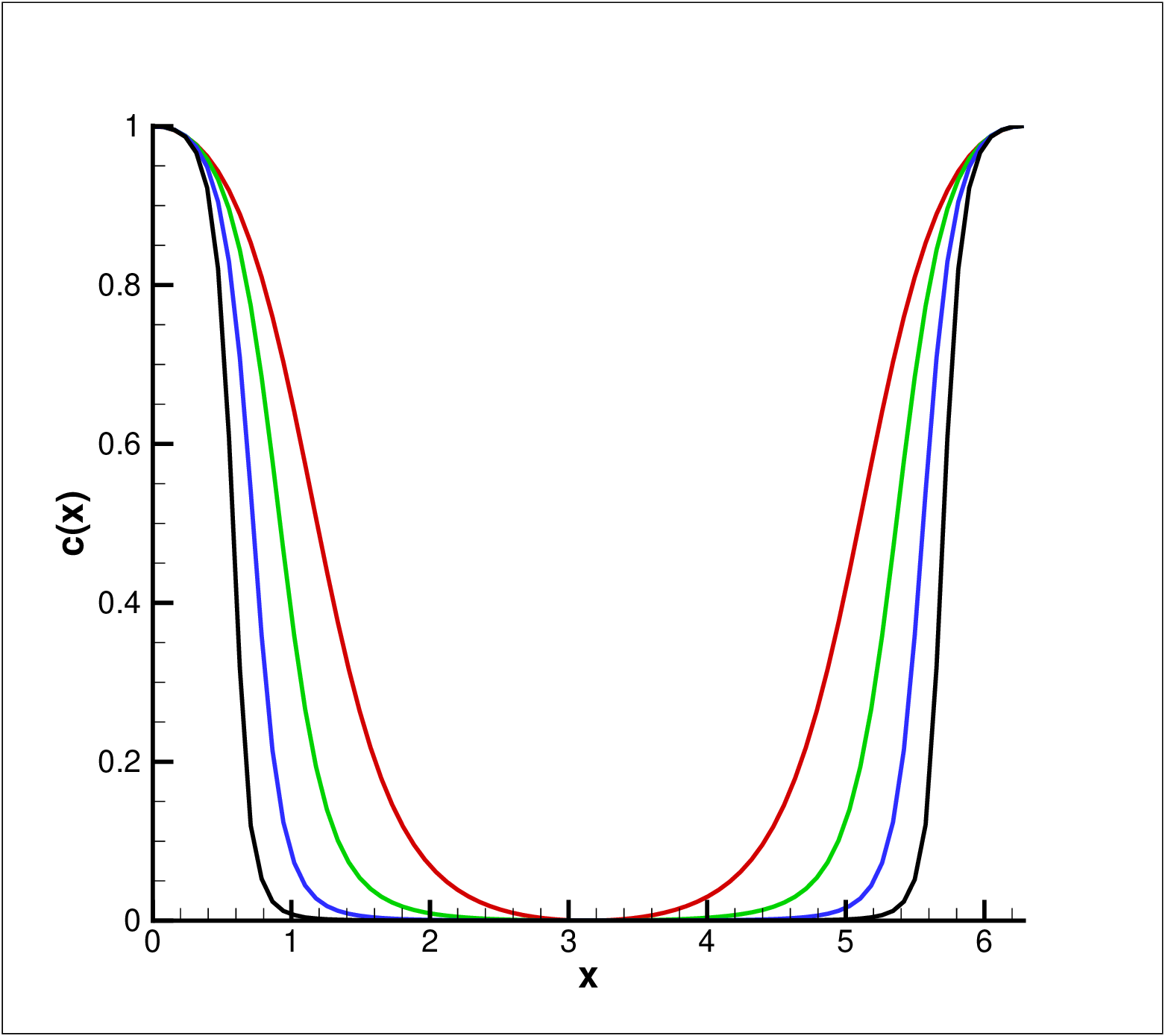}
    \includegraphics[width=0.35\textwidth]{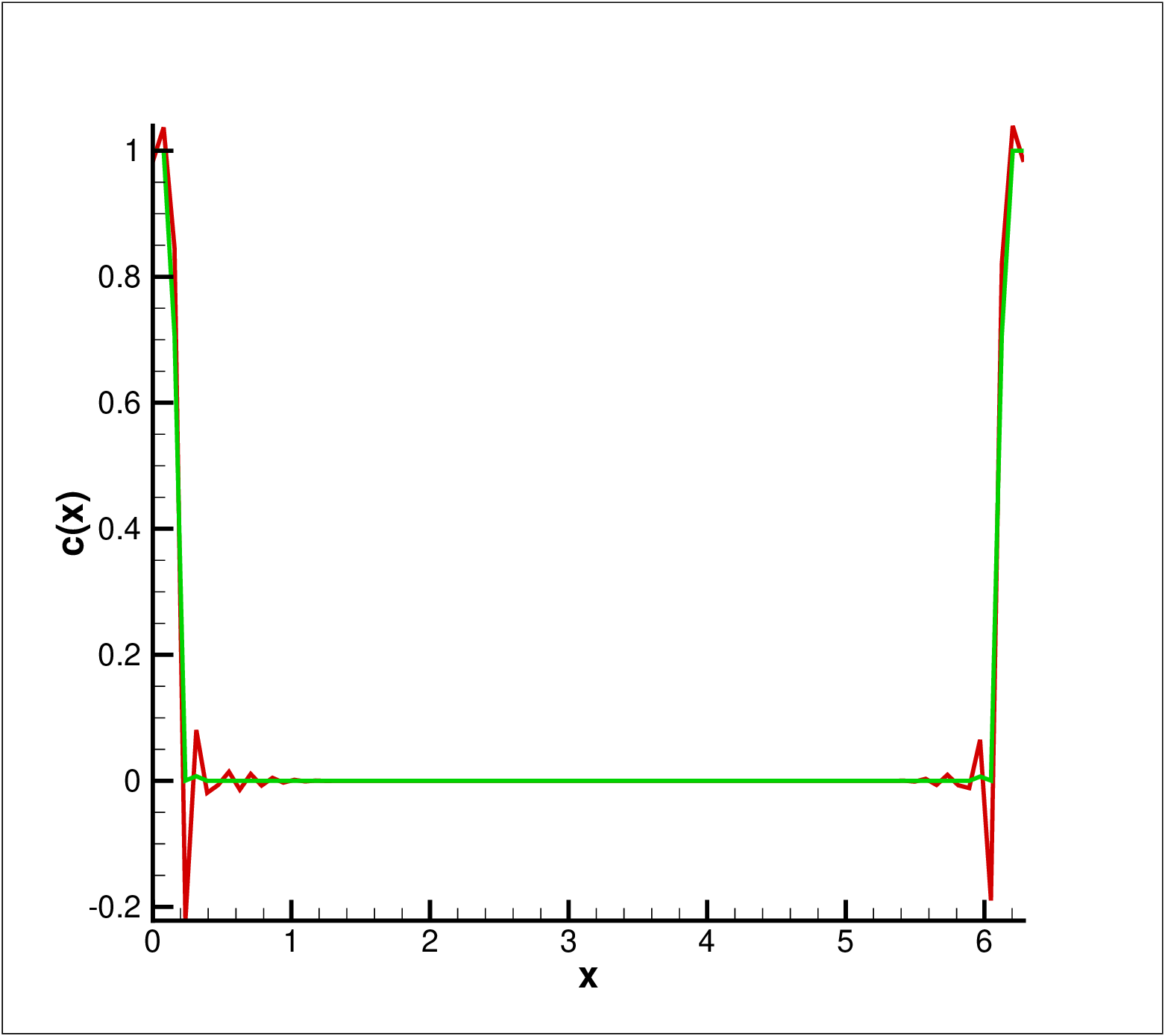}
    \caption{Example \ref{ex2}: Left panel shows numerical approximations of $c$ for $\gamma$=1 (red), 2 (green), 3 (blue), 4 (black) with bound-preserving limiter. Right panel shows numerical approximations of $c$ for $\gamma$=10 with (green) and without (red) bound-preserving limiter. Other parameters are taken as $t=0.1$ and $N=80.$}\label{ex2_figure}
\end{figure}
From the figure, we can observe that the larger the $\gamma$, the larger the gradient in the numerical approximation. Next, we take $\gamma=10$ to test the effect of the bound-preserving limiter. In the right panel of Figure \ref{ex2_figure}, the green and red curves are the numerical solutions obtained with and without bound-preserving technique, respectively. It is easy to see that, with the special technique, the green curve is not oscillatory and the numerical approximation lies between $0$ and $1$. This example clearly demonstrate that our algorithm can also be applied to problems with the present
of vacuum.

In the previous two examples, the numerical solutions with and without bound-preserving limiters look similar. However, in the following example, we can observe significant differences. Finally, we consider the necessity of the bound-preserving limiter

\begin{example}\label{ex3}
We choose the initial condition as
$$
c(x,0)=\left\{\begin{array}{cc}1,&x<1\\0,&x>1\end{array}\right.,\quad\quad p(x,0)=\left\{\begin{array}{cc}5,&x<1\\0,&x>1\end{array}\right..
$$
Other parameters are chosen as
$$
q(x,t)=0,\ z_1=0.1,\ \kappa(x)=\mu(c)=z_2=1,\ \phi(x)=1,\ D(u)=0,
$$
\end{example}
We compute up to $T=1$. with $\Delta t=0.001\Delta x^2$ to reduce the time error. We solve the problem with the bound-preserving limiter and the result is given in Figure \ref{ex3_figure}. We can see that the numerical approximation lies between 0 and 1.
\begin{figure}[!htpb]\centering
    \includegraphics[width=0.35\textwidth]{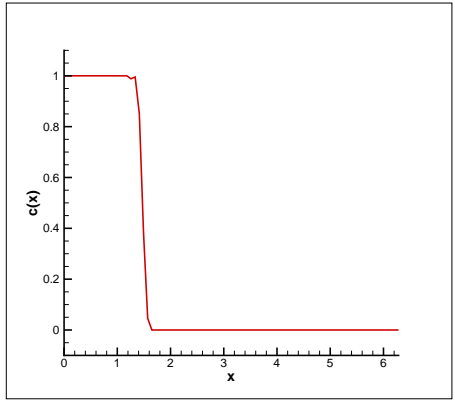}
    \caption{Example \ref{ex3}: Numerical approximations of $c$ at $t=1$.}\label{ex3_figure}
\end{figure}
Moreover, we also solve the problem without the bound-preserving limiter and the numerical approximation blows up at $t\approx0.19$. This is because, near $x=1$, the gradient of $c$ is large, which further yields strong oscillations in the numerical approximations. Therefore, the value of $d(c)=1-0.9c$ might be small or even negative, leading to ill-posed problems. We also choose a smaller time step, say $\Delta t=0.0001\Delta x^2$, and the numerical approximations blow up at $t\approx0.106$.

\subsection{Two space dimensions}
In this subsection, we consider \eqref{origin} subject to boundary condition \eqref{boundary2d}. We solve the problems on the computational domain $\Omega=[0,2\pi]\times[0,2\pi]$. If not otherwise stated, we take $N_x=N_y=N.$ We first construct an analytical solution and test the accuracy of the bound-preserving DG method.
\begin{example}\label{accuracy2d}
We choose the initial condition as
$$
c(x,y,0)=\frac12(1-\cos(x)\cos(y)),\quad\quad p(x,0)=\cos(x)\cos(y)-1.
$$
The source parameters $q$ and $\tilde{c}$ are taken as
$$
q(x,y,t)=2e^{-2t},\quad\tilde{c}=\frac12\left(e^{-2\gamma t}\left(\frac12\sin^2(x)\cos^2(y)+\frac12\cos^2(x)\sin^2(y)-\cos(x)\cos(y)\right)+1\right).
$$
Moreover, we also choose other parameters as
$$
z_1=z_2=1,\quad\phi(x,y)=\kappa(x,y)=\mu(c)=1,\quad {\bf D}({\bf u})=\left(\begin{array}{cc}\gamma&0\\0&\gamma\end{array}\right).
$$
\end{example}
It is easy to see that the exact solutions are
$$
c(x,y,t)=\frac12(1-e^{-2\gamma t}\cos(x)\cos(y)),\quad\quad p(x,y,t)=e^{-2t}(\cos(x)\cos(y)-1).
$$
We choose $\Delta t=0.02\min\{\Delta x^2,\Delta y^2\}$ with final time $T=0.01$. We take $\gamma=10^{-3}$ such that the exact solution $c$ is very close to $0$ for $t>0$. We compute the $L^\infty$-norm of the error between the numerical and exact solution, and results are given in Table \ref{accuracy_test2d}.
\begin{table}[!htbp]
\centering
\begin{tabular}{c|c|c|c|c}
\hline
&\multicolumn{2}{c|}{with limiter}&\multicolumn{2}{c}{ no limiter}\\\hline
$N$ & $L^{\infty}$ error & order&$L^{\infty}$ error & order
\\\hline
20 &1.04e-2& -- &1.00e-2& -- \\%30.96%
40 &2.64e-3&1.97&2.53e-3&1.99\\%15.31%
80 &6.77e-4&1.96&6.36e-4&1.99\\%7.60%
160&1.75e-4&1.96&1.61e-4&1.99\\%3.78%
\hline
\end{tabular}
\caption{Example \ref{accuracy2d}: accuracy test at $T=0.1$ for the second-order DG methods with and without the bound-preserving limiter.}
\label{accuracy_test2d}
\end{table}
We can observe second-order accuracy of the DG scheme with and without bound-preserving limiter. Therefore, the limiter does not kill the accuracy.

\begin{example}\label{ex5}
We choose the initial condition as
$$
c(x,y,0)=\left\{\begin{array}{ll}1,& 0<x<1,\ 0<y<1\\0,&otherwise\end{array}\right.,\quad\quad p(x,y,0)=\left\{\begin{array}{ll}5,& 0<x<1,\ 0<y<1\\0,&otherwise\end{array}\right..
$$
Other parameters are taken as
$$
q(x,y,t)=0,\ z_1=0.1,\ z_2=1,\ \phi(x,y)=1,\ {\bf D}({\bf u})={\bf 0}.
$$
\end{example}
We compute the numerical approximations of $c$ at $T=0.1$ and $T=0.5$, with $\Delta t=0.001\min\{\Delta x^2,\Delta y^2\}$ and the bound-preserving limiter. The results are given in Figure \ref{ex5_figure}. We can see that the numerical approximation lies between 0 and 1.
\begin{figure}[!htpb]\centering
    \includegraphics[width=0.35\textwidth]{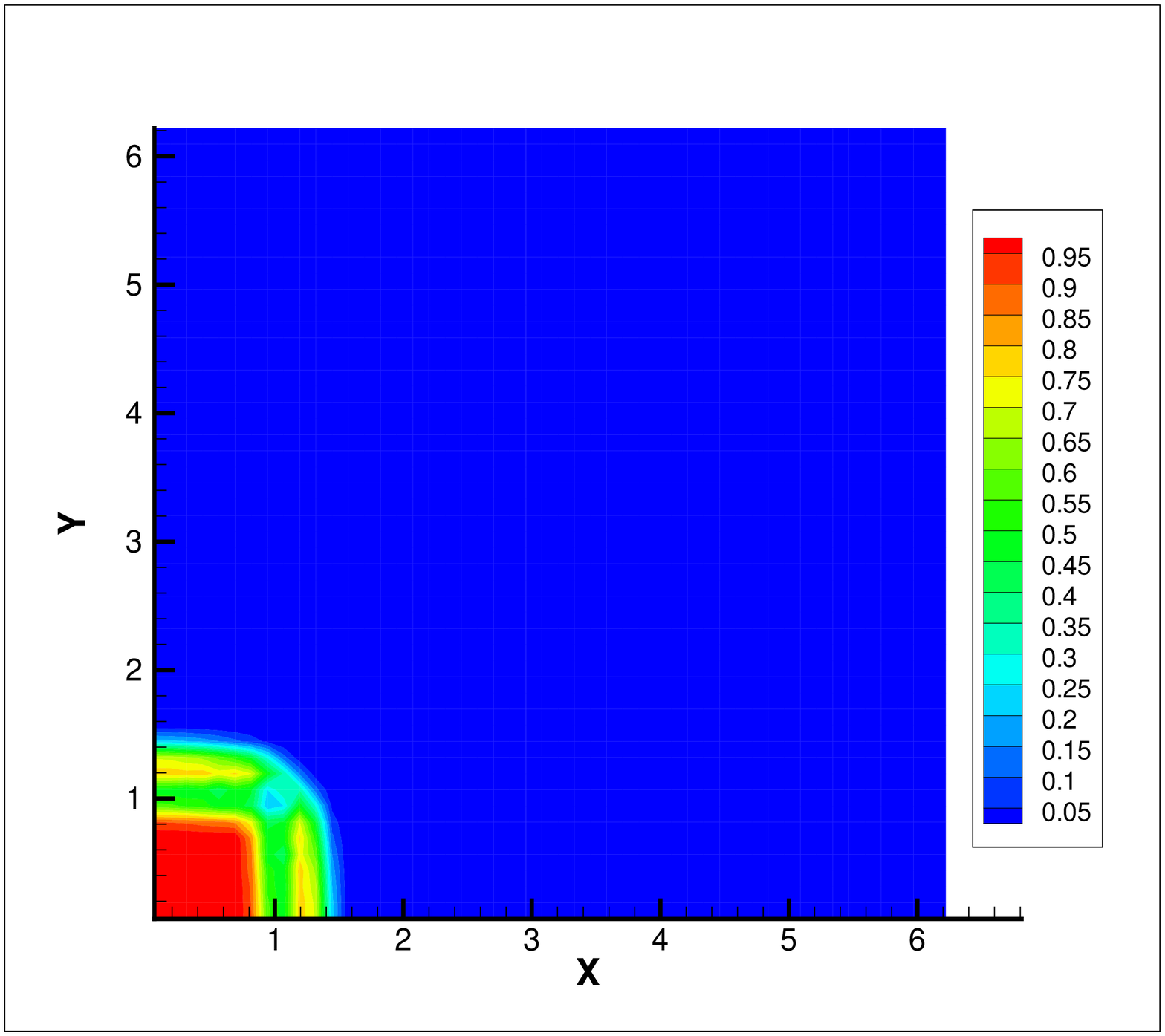}
    \includegraphics[width=0.35\textwidth]{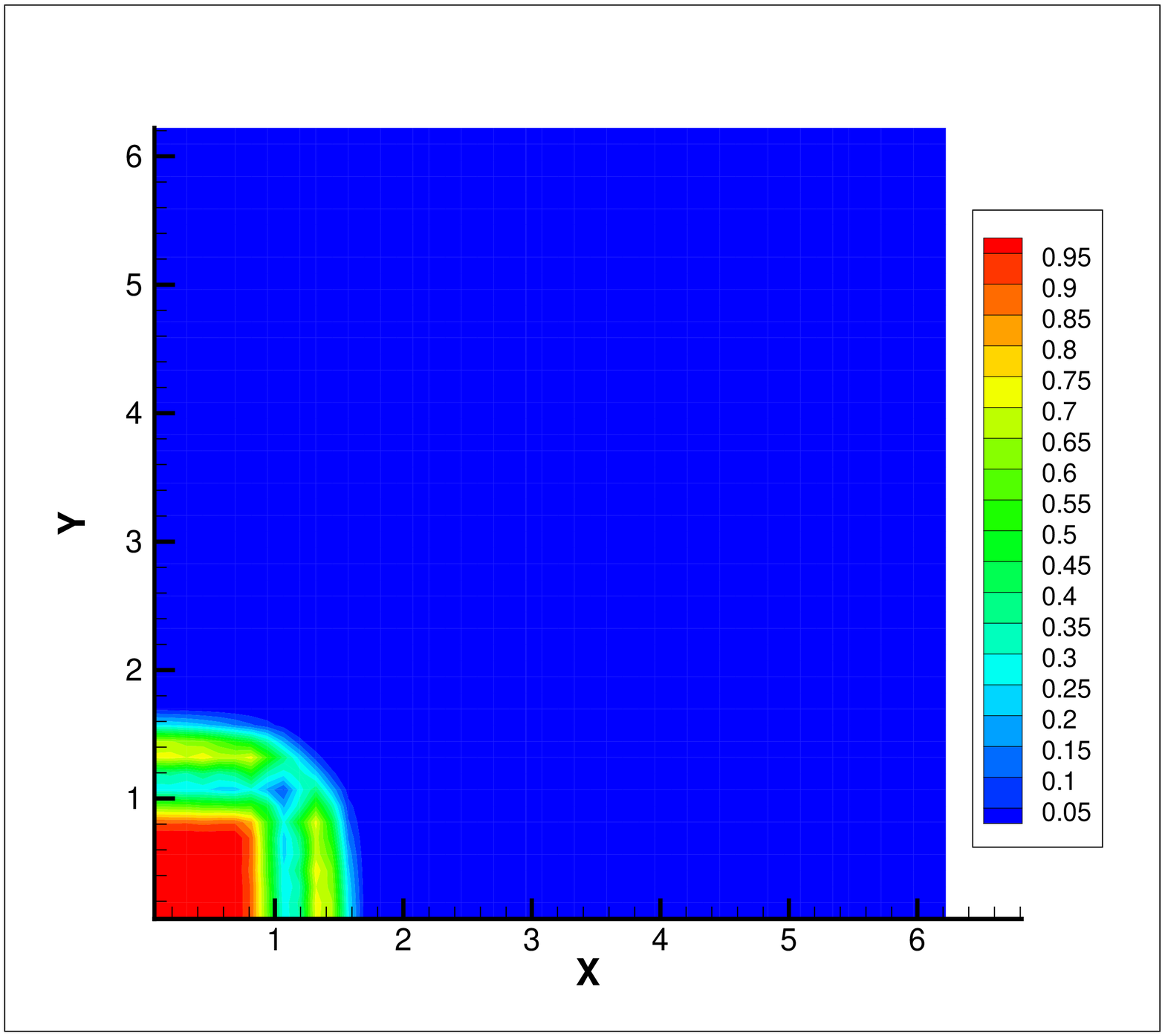}
    \caption{Example \ref{ex5}: Numerical approximations of $c$ at $t=0.1$ (left) and $t=0.5$ (right).}\label{ex5_figure}
\end{figure}
However, without the bound-preserving limiter, the numerical approximation blows up at $t\approx1.89\times10^{-3}$ due to the ill-posedness of the problems.

\begin{example}\label{ex6}
We choose the initial condition as
$$
c(x,y,0)=0.5,\quad\quad p(x,y,0)=0.
$$
Other parameters are taken as
$$
\ z_1=0.4,\ z_2=0.6,\ \phi(x)=1,\ {\bf D}({\bf u})=\left(\begin{array}{cc}|u|&0\\0&|u|\end{array}\right).
$$
The injection well is located at $(2\pi,2\pi)$ with $\tilde{c}=1$ while the production well is located at $(0,0)$.
\end{example}
We choose $\Delta t=0.01\min\{\Delta x^2,\Delta y^2\}$ and compute the numerical approximation of $c$ at t=1, 10 and 100. The results are given in Figure \ref{ex6_figure}.
\begin{figure}[!htp]\centering
    \includegraphics[width=0.32\textwidth]{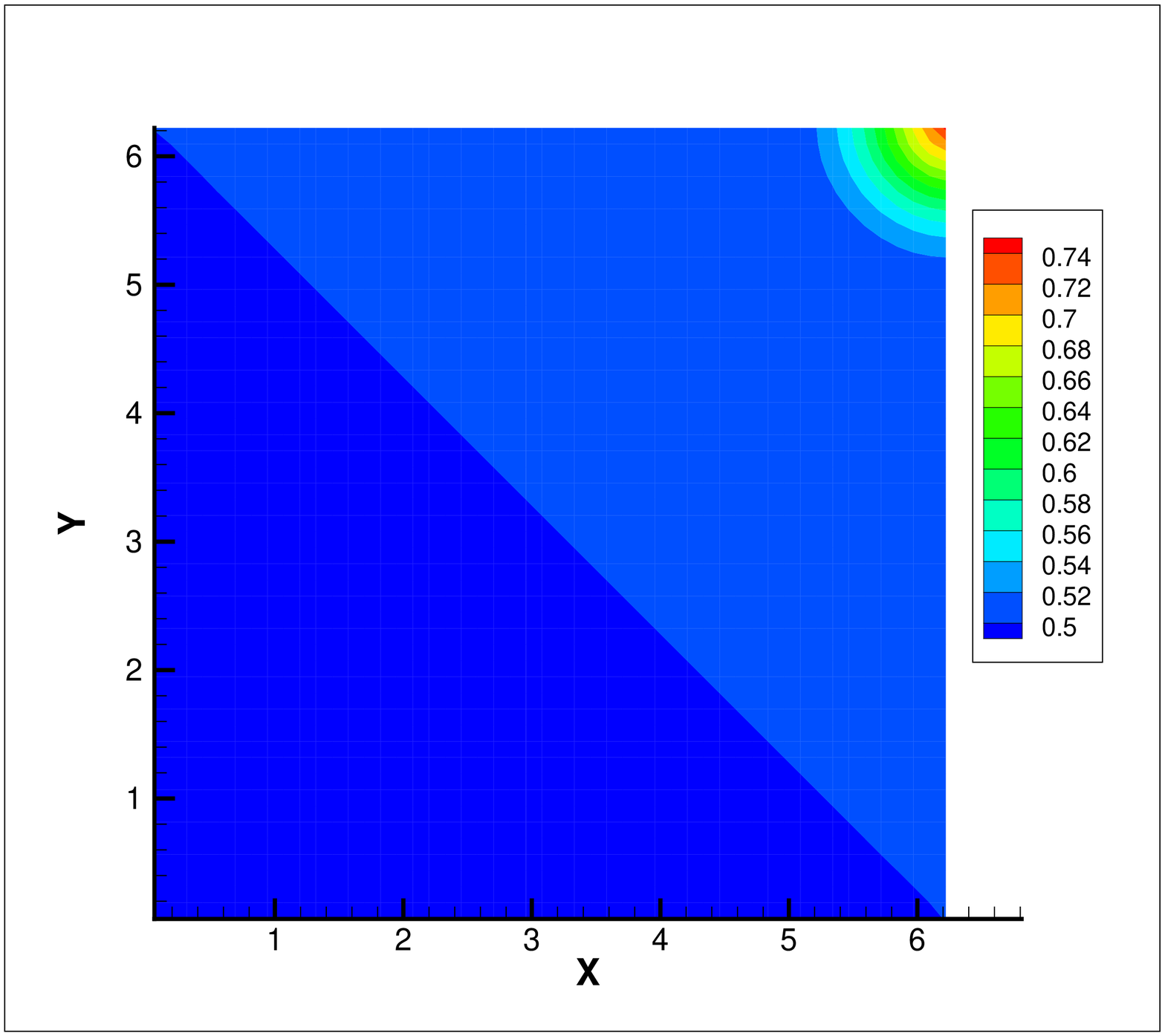}
    \includegraphics[width=0.32\textwidth]{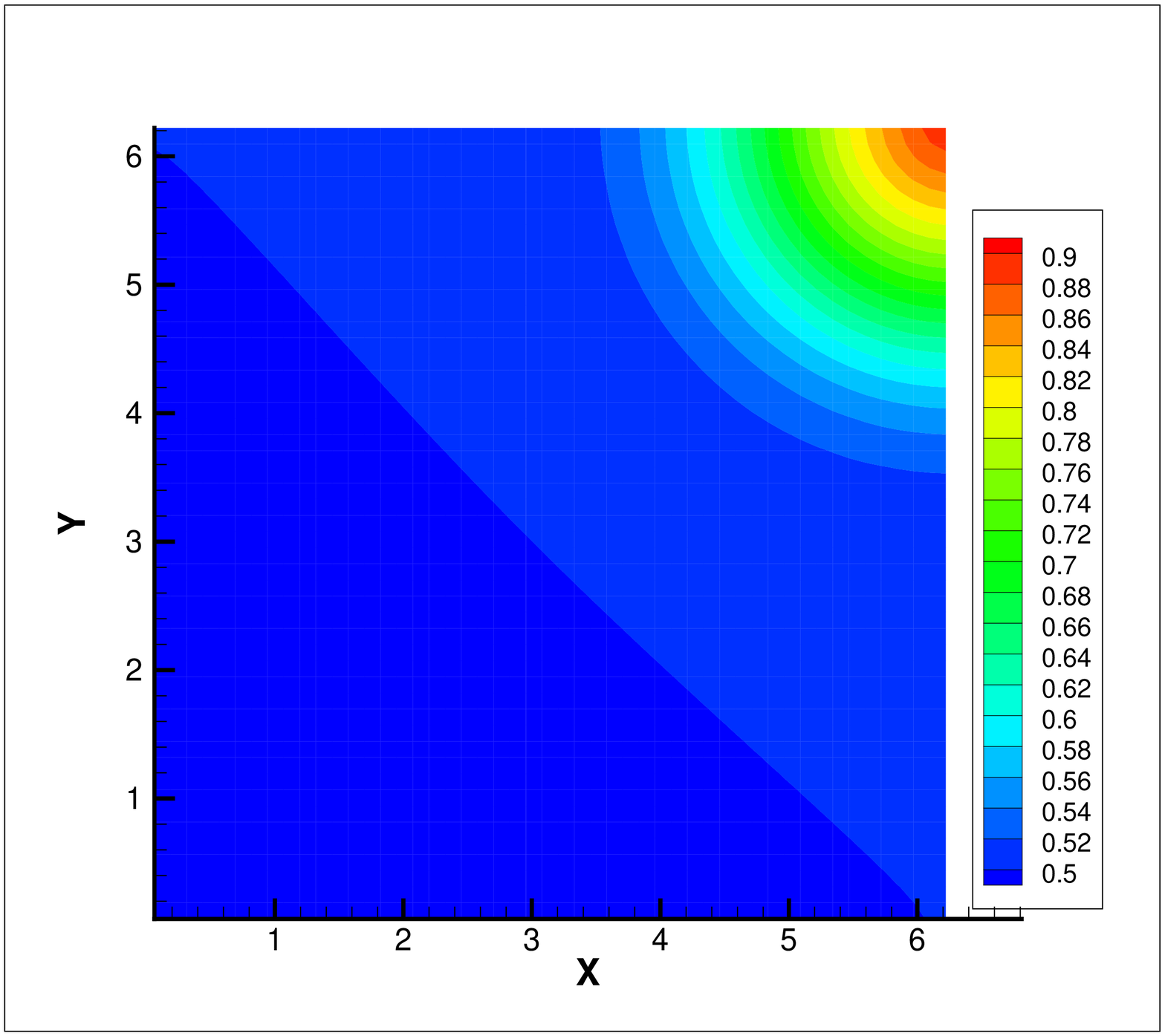}
    \includegraphics[width=0.32\textwidth]{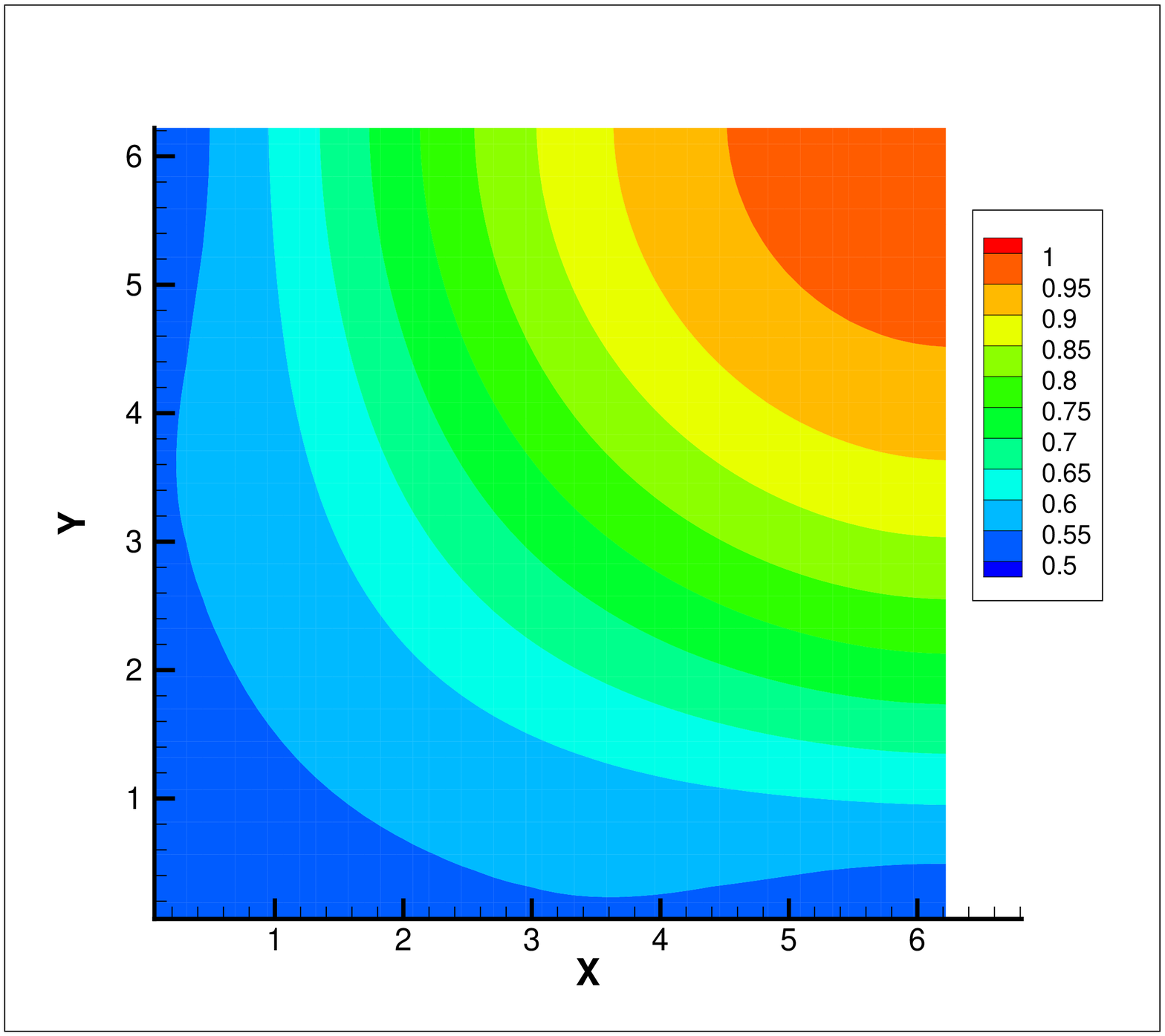}
    \caption{Example \ref{ex6}: Numerical approximations of $c$ at $t=1$ (left) $t=10$ (middle) $t=100$ (right) with $N_x=N_y=50$ and bound-preserving limiter.}\label{ex6_figure}
\end{figure}
We can observe that all the numerical approximations are between $0$ and $1$. Therefore, the bound-preserving technique also works for the 2D problems.

\section{Concluding remarks}
\label{sec5}
In this paper, we apply DG methods to compressible miscible displacement problem in porous media. The bound-preserving technique has been applied to the problems in one and two space dimensions. In the future, we will apply the technique to triangular meshes and consider multi-species fluid mixtures.

\end{document}